\documentclass[11pt]{article}
\usepackage{amsthm, amsmath, amssymb, amsfonts, url, booktabs, tikz, setspace, fancyhdr, amsbsy}
\usepackage[margin = 0.8in]{geometry}
\usepackage{hyperref, enumerate}
\usepackage{caption,float}
\usepackage{subcaption}
\usepackage{esint}
\usepackage{verbatim}
\usepackage{multirow}

\newtheorem{theorem}{Theorem}[section]
\newtheorem{proposition}[theorem]{Proposition}
\newtheorem{lemma}[theorem]{Lemma}

\theoremstyle{definition}
\newtheorem{definition}[theorem]{Definition}
\newtheorem{assumption}[theorem]{Assumption}
\newtheorem{remark}[theorem]{Remark}

\DeclareMathOperator*{\argmin}{arg\,min}

\newcommand{\norm}[1]{\left\lVert#1\right\rVert}

\newcommand{\R}{\mathbb{R}}
\newcommand{\N}{\mathbb{N}}
\newcommand{\V}{\mathbb{V}}
\newcommand{\G}{\mathcal{G}}
\newcommand{\LL}{\mathcal{L}}
\newcommand{\HH}{\mathcal{H}}
\newcommand{\RR}{\mathcal{R}}
\newcommand{\EE}{\mathbb{E}}
\newcommand{\defeq}{\mathrel{\mathop:}=}

\newcommand{\dx}{\,\mathrm{d}x}

\newcommand{\NN}{\mathcal{N}}
\newcommand{\aaa}{\boldsymbol{a}}

\newcommand{\bb}{\boldsymbol{b}}
\newcommand{\om}{\omega}
\newcommand{\dom}{\,\mathrm{d}\omega}

\numberwithin{equation}{section}

\def\ocirc#1{\ifmmode\setbox0=\hbox{$#1$}\dimen0=\ht0
    \advance\dimen0 by1pt\rlap{\hbox to\wd0{\hss\raise\dimen0
    \hbox{\hskip.2em$\scriptscriptstyle\circ$}\hss}}#1\else
    {\accent"17 #1}\fi}

\begin{document}

\title{Error analysis for finite element operator learning methods for solving parametric second-order elliptic PDEs}

\author{Youngjoon Hong\thanks{Department of Mathematical Sciences, KAIST, Daejeon, Republic of Korea. Email: \tt{hongyj@kaist.ac.kr}} , ~Seungchan Ko\thanks{Department of Mathematics, Inha University, Incheon, Republic of Korea. Email: \tt{scko@inha.ac.kr}}
,~and ~Jaeyong Lee\thanks{Department of AI, Chung-Ang University, Seoul, Republic of Korea. Email: \tt{jaeyong@cau.ac.kr}}}

\date{~}

\maketitle

~\vspace{-1.5cm}

\begin{abstract}
In this paper, we provide a theoretical analysis of a type of operator learning method without data reliance based on the classical finite element approximation, which is called the finite element operator network (FEONet). 
We first establish the convergence of this method for general second-order linear elliptic PDEs with respect to the parameters for neural network approximation. 
In this regard, we address the role of the condition number of the finite element matrix in the convergence of the method. Secondly, we derive an explicit error estimate for the self-adjoint case. 
For this, we investigate some regularity properties of the solution in certain function classes for a neural network approximation, verifying the sufficient condition for the solution to have the desired regularity. Finally, we will also conduct some numerical experiments that support the theoretical findings, confirming the role of the condition number of the finite element matrix in the overall convergence.
\end{abstract}


\smallskip


\section{Introduction}\label{sec_intro}
%Numerical methods play a crucial role in approximating solutions for partial differential equations (PDEs), particularly when exact solutions are unfeasible for complex systems. 
%Among the various techniques—finite difference, finite element, and finite volume methods—the finite element method (FEM) is prominently used in engineering and physics \cite{FEM_book1, FEM_book2}.
%Significant advancements in the mathematical analysis of FEM have notably enhanced its reliability \cite{FEM_book3, FEM_book4}. 
%Its ability to adeptly handle irregular geometries and complex boundary conditions makes FEM invaluable for diverse applications, including structural analysis, heat transfer, fluid dynamics, and electromagnetic studies \cite{FEM_book5, FEM_book6}. 
%However, the application of these numerical methods, including FEM, is often accompanied by considerable computational cost.
%In this paper, we provide a convergence analysis of a type of scientific machine learning based on the classical finite element approximation.
The emerging field of scientific machine learning, which bridges the gap between traditional numerical analysis and machine learning, has introduced innovative approaches that enrich conventional numerical methods, especially in tackling complex tasks. At the forefront of this evolution is the field of physics-informed neural networks (PINNs) \cite{pinn01}. PINNs employ neural networks trained to comprehend the underlying physics of systems, thereby enhancing the capability to solve PDEs for physics-based problems using neural networks.
%\noteKo{thereby significantly improving our ability to solve PDEs (would this sentence be okay for the classical FE guys?)} 
This advancement has spurred the development of various PINN variants \cite{MR4554720,yang2021b}. However, these variants are limited by their need for retraining with each new set of input data, such as initial conditions and boundary conditions, which hampers their utility in dynamic systems where real-time predictions are essential.

Addressing this limitation, operator networks have emerged, employing data-driven approaches to understand mathematical operators in physical systems, particularly for parametric PDEs \cite{li2023fourier, boulle2023mathematical}. 
A significant breakthrough in this area is the Deep Operator Network (DeepONet) architecture \cite{lu2021learning}, founded on the universal approximation theorem for operators. DeepONets facilitate rapid solution prediction when PDE data varies, but depends on extensive pre-computed training data pairs, which is a demanding task, especially for complex or nonlinear systems.
To mitigate these issues, hybrid models like Physics-Informed Neural Operator (PINO) \cite{li2021physics} and Physics-Informed DeepONet (PIDeepONet) \cite{wang2021learning} have been introduced. 
These models amalgamate the strengths of PINNs and operator learning by embedding physical equations within the loss function of neural operators. Despite this innovation, they often face challenges such as reduced accuracy in complex geometries, difficulty managing stiff problems, and significant generalization errors due to limited input data \cite{jagtap2020extended,costabal2024delta,lee2023hyperdeeponet}.
Furthermore, employing neural networks as the solution space complicates the imposition of various boundary values, consequently impacting the precision of solutions \cite{choi2024spectral}.

In response to these limitations, a novel unsupervised operator network based on finite element methods, termed the Finite Element Operator Network (FEONet), was developed \cite{FEONet}. 
The main focus of this paper is on the extensive error analysis of FEONet. 
In the finite element method (FEM) framework, the numerical solution $u_h(x)$ is approximated as a linear combination of nodal coefficients $\alpha_k$, and nodal basis functions, $\phi_k(x)$, defined by piecewise polynomials over a mesh. This is represented as
$
u_h(x) = \sum \alpha_k \phi_k(x), x \in \mathbb{R}^d.
$
Building on this concept, the FEONet predicts PDE solutions under various inputs like initial conditions and boundary conditions. It is versatile and able to handle multiple PDE instances across complex domains as shown in Figure \ref{fig:meshes} without data reliance. 
The loss function of FEONet, inspired by the classical FEM, is based on the residual of the finite element approximation, ensuring accurate PDE solutions and exact compliance with boundary conditions. 
The FEONet approach infers coefficients, $\widehat{\alpha}_k$, for constructing the linear combination $\sum\widehat{\alpha}_k\phi_k$ to approximate PDE solutions. 
Thanks to the ability of FEM to handle boundary conditions, solutions predicted by FEONet also satisfy exact boundary conditions. 
Notably, its unique feature lies in solving parametric PDEs without relying on any paired input-output training data, a significant step forward in computational efficiency and application versatility. Another advantage of the FEONet is its applicability to singularly perturbed problems. 
By integrating the boundary layer element into the finite element space using the corrector basis function, we can establish an enriched basis scheme for utilization in the FEONet, allowing the model to capture sharp transitions accurately; see \cite{FEONet} for more details.

\begin{figure}
\begin{center}
\includegraphics[width=0.9\textwidth]{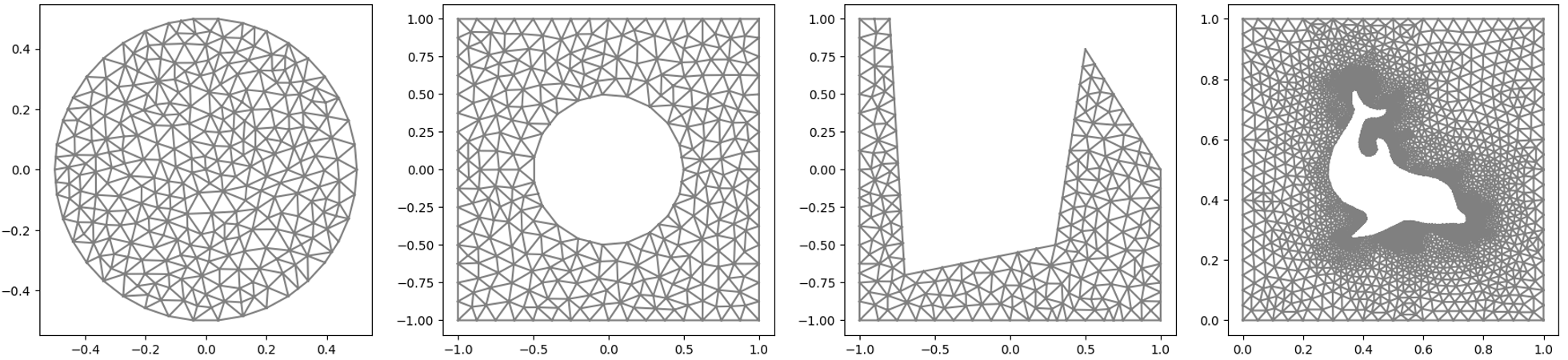}
\end{center}
\caption{Examples of complex domains where FEONet can predict solutions.}
\label{fig:meshes}
\end{figure}

While significant progress has been made in scientific machine learning, a notable gap remains in the area of rigorous convergence analysis, which is essential for establishing the reliability of these novel methodologies \cite{marcati2023exponential, elbrachter2022dnn, de2024wpinns}. 
Some analytical progress has been made with the Deep Ritz method \cite{marwah2023neural,rade_upper_3, rade_upper_2}, PINNs \cite{shin2020convergence, mishra2023estimates}, and Operator Network approaches \cite{lanthaler2022error, kovachki2021universal}. 
However, discrepancies between theoretical results and numerical experiment outcomes frequently arise, highlighting the need for more comprehensive and aligned studies in this area. In addition, to date, there has been a lack of convergence analysis for unsupervised operator networks such as the FEONet method.
In this regard, a key part of our contribution is integrating the well-established FEM theory within the FEONet framework. This integration enables a more grounded and consistent convergence analysis, closely aligned with the outcomes of computational experiments.
Our research focuses on this convergence analysis, bridging the gap between theoretical rigor and practical application. 
It highlights the potential of FEONet as a reliable tool for solving second-order elliptic PDEs, demonstrating its suitability for complex computational challenges.

The primary objective of the present paper is to perform a theoretical analysis of the FEONet, addressing the convergence of the method and deriving an error estimate. To be specific, we consider the general second-order linear elliptic PDE of the form
\begin{align}
    -\,{\rm{div}}\,(\aaa(x)\nabla u)+\bb(x)\cdot\nabla u+c(x)u&=f(x)\quad{\rm{in}}\,\,D,\label{eq1} \\
u(x)&=g(x)\quad{\rm{on}}\,\,\partial D.\label{eq2}
\end{align}
A similar type of theoretical analysis was performed in \cite{ULG_anal, FEONet}. Note, however, that the result in \cite{ULG_anal, FEONet} is restricted to the case of self-adjoint equations, and only a partial convergence result was obtained. In this paper, we will consider more general types of equations and investigate the comprehensive convergence properties of the FEONet. The main novelty of the present paper can be summarized as follows:
In Section \ref{sec:conv}, we prove the convergence of the predicted solution by the FEONet. Our analysis is applicable to a general class of equations, which generalizes the previous results in \cite{ULG_anal, FEONet}. In the proof of convergence, we see that the convergence of the FEONet is influenced by the condition number of the finite element matrix, which is also confirmed by the numerical experiments. 
Subsequently in Section \ref{sec:err_est}, unlike the previous papers where the parameters of baseline numerical methods were fixed, we identify the role of these parameters in overall convergence and conduct the complete theoretical analysis of the method. 
Additionally, in certain scenarios, we derive explicit error estimates for FEONet, utilizing a novel regularity theory developed for our method. Section \ref{sec:num_exp} then presents numerical experiments that validate our theoretical findings. Furthermore, guided by the theoretical results obtained in the previous sections, we utilize the preconditioning techniques from the numerical analysis and confirm their significant impact on convergence and training efficiency.

In the next section, as a starting point, we briefly describe the idea of FEONet and demonstrate how it is trained and it predicts the solution within a suitable finite element setting. 
We also collect some results for the eigenvalue estimates of the finite element matrices which play an important role in the entire analysis. 
Additionally, we shall set up an analytic background for the theoretical analysis as well as the mathematical description of neural networks. 
Then the main sections of the paper proceed as described above, and some concluding remarks are made at the end of the paper.

\section{Framework for finite element operator networks}\label{sec:prelim}

In the equations \eqref{eq1}-\eqref{eq2}, we shall assume that 
\begin{equation}\label{coef_cond_1}
\aaa\in L^{\infty}(D)^{d\times d},\, \bb\in W^{1,\infty}(D)^d,\, c\in L^{\infty}(D)\,\,{\rm{and}}\,\, f\in H^{-1}(D). \end{equation}
For the diffusion coefficients $\aaa=(a_{ij})$, we also assume that the uniform ellipticity condition holds: there exists a positive constant $\tilde{a}>0$ such that
\begin{equation}\label{coef_cond_2}
    \sum^d_{i,j=1}a_{ij}(x)\xi_i\xi_j\geq\tilde{a}\sum^d_{i=1}\xi^2_i,\quad\forall\xi=(\xi_1,\cdots,\xi_d)\in\R^d,\quad x\in \overline{D}.
\end{equation}
For the well-posedness of the equation \eqref{eq1}-\eqref{eq2} we further assume that
\begin{equation}\label{coef_cond_3}
    c(x)-\frac{1}{2}\,{\rm{div}}\,\bb(x)\geq 0,\quad x\in\overline{D}.
\end{equation}
The corresponding weak formulation is defined as follows: find $u\in H^1_0(D)$ such that
\[
B[u,v]:=\int_{D}\aaa(x)\nabla u\cdot\nabla v\dx+\int_{D}\bb(x)\cdot\nabla u v\dx+\int_{D}c(x)uv\dx=\int_D
f(x)v\dx=:l(v)\quad\forall v\in H^1_0(D).
\]
Thanks to the assumptions \eqref{coef_cond_1}, \eqref{coef_cond_2} and \eqref{coef_cond_3}, there holds for some constants $c_0$, $c_1$ and $c_2>0$ that
\begin{equation}\label{lax_cond}
B[v,v]\geq c_0\|v\|_{H^1(D)},\,\,|B[u,v]|\leq c_1\|u\|_{H^1(D)}\|v\|_{H^1(D)},\,\,{\rm{and}}\,\,|\ell(v)|\leq c_2\|v\|_{H^1(D)},
\end{equation}
and the existence of a unique weak solution follows by the standard Lax--Milgram theory (see, e.g., \cite{BS_book}).

\subsection{Finite element operator networks}\label{subsec:desc}
In this section, we aim to describe FEONet, the main numerical scheme under consideration. This is a novel method introduced in \cite{FEONet}, which utilizes the approximation power of deep neural networks in conjunction with the classical finite element approximation. More precisely, this method leverages the FEM for neural networks to learn the solution operator without any paired input-output training data to solve a wide range of parametric PDEs.

As a first step, let us define the finite element space, which will be used throughout the paper. Let $\G_h$ be a shape-regular partition of a given physical domain $\overline{D}$, where $h_E$ denotes the diameter of $E\in\G_h$ and $h=\max_{E\in\G_h}h_E$. We will also assume that there exists a positive constant $\gamma>0$ independent of $h>0$ such that $\max_{E\in\G_h}\frac{h_E}{\rho_E}\leq\gamma$, where $\rho_E$ is the supremum of the diameters of inscribed balls for an element $E\in\G_h$. For a given partition $\G_h$, the finite element spaces are defined by
$\V_h=\V(\G_h)\defeq \{V\in C(\overline{
D}):V_{|E}\in\hat{\mathbb{P}}_{\V},E\in \G_h\,\,\text{and}\,\,V_{|\partial D}=0\}$,
\begin{comment}
\begin{itemize}
\item {\bf{Affine equivalence}}: For each element $E\in \G_h$, there exists a non-singular affine mapping $\boldsymbol{F}_E:E\rightarrow\hat{E}$,
where $\hat{E}$ is the reference $d$-simplex in $\R^d$.
\item {\bf{Shape-regularity}}: For an element $E\in\G_h$, let $\rho_E$ be the supremum of the diameters of inscribed balls. Then there exists a positive constant $\gamma>0$ independent of $h>0$ such that $\max_{E\in\G_h}\frac{h_E}{\rho_E}\leq\gamma.$
\end{itemize}
\end{comment}
where $\hat{\mathbb{P}}_{\V}\subset W^{1,\infty}(\hat{E})$ is a finite-dimensional subspace. Here we assume that $\V_h$ has finite and locally supported basis; e.g., for each $h>0$, there exists $N_h\in\mathbb{N}$ such that $\V_h=\text{span}\{\phi^h_1,\ldots,\phi^h_{N_h}\}$ and for any basis function $\phi^h_i$, $i=1,\ldots,N_h$, we have that if $\phi^h_i\neq0$ on $E$ for some $E\in \G_h$, then $\text{supp}\,\phi^h_j\subset\bigcup\{E'\in G_h:E'\cap E\neq\emptyset\}=: S_E$. Furthermore, by shape regularity, it follows that $\exists C\in\mathbb{N}$ such that $ |S_E|\leq C|E|\,\,\,{\rm{for}}\,\,\,{\rm{all}}\,\,\,E\in\G_h$, where $C$ can be chosen to be independent of $h$. 

Next, we define our Galerkin approximation. We seek a discrete solution $u_h = \sum^{N_h}_{k=1}\alpha_k\phi_k\in\V_h$ satisfying 
\begin{equation}\label{var_for}
B[u_h,v_h]=\ell(v_h)\quad{\text{for all}}\,\,v_h\in \V_h.
\end{equation}
If we write $S=(S_{ij})_{1\leq i,\,j\leq N_h}$, $C=(C_{ij})_{1\leq i,\,j\leq N_h}\in\R^{N_h\times N_h}$, $F=(F_j)_{1\leq j\leq N_h}\in\R^{N_h}$ with
\begin{equation}\label{vec_form}
    S_{ij}=\int_D\left[ \boldsymbol{a}(x)\nabla\phi_i\cdot\nabla\phi_j+c(x)\,\phi_i\,\phi_j\right]\dx,\quad
    C_{ij}=\int_D ({\boldsymbol{b}}(x)\cdot\nabla\phi_i)\phi_j\dx,\quad
    F_j=\int_D f(x)\,\phi_j\dx,
\end{equation}
then the Galerkin approximation is equivalent to the linear algebraic equations
\begin{equation}\label{LAS}
    (S+C)\alpha=F,\quad {\rm{with}}\,\,\,\alpha=(\alpha_k)_{1\leq k\leq N_h}\in\R^{N_h}.
\end{equation}

Now, let us introduce the finite element operator network, proposed in \cite{FEONet}. As mentioned earlier, the input of the FEONet can be any type of PDE data, for example, external force, variable coefficient, or boundary condition. Here, as a prototype example, we shall consider the networks whose input is an external forcing term. However, this can be extended to different types of input data in a straightforward manner. For each external forcing term $f$, instead of computing the coefficients using \eqref{LAS}, we approximate the coefficients $\alpha$ using deep neural networks. For this purpose, we set the input of neural networks as external force $f$, parametrized by a random parameter $\om$ for the (possibly high-dimensional) probability space $(\Omega,\mathcal{T},\mathbb{P})$, assuming that $\Omega$ is compact. Typical examples are the Gaussian random fields or the random forcing terms defined by $f(x,\om)=\om_1\sin(2\pi\om_2x)+\om_3\cos(2\pi\om_4x)$,
where $\om=(\om_1,\om_2,\om_3,\om_4)$ is i.i.d. distributed uniformly with $\omega_j \in [0,1]$. In the present paper, we shall interpret $f(x,\om)$ as a Bochner-type function defined on $D\times\Omega$, and we will assume the following throughout the paper.

\begin{assumption}\label{f_ass}
    If we use the $(P \ell)$-finite element approximation (piecewise polynomial function of degree less than or equal to $\ell\in\mathbb{N}$), we assume that $f(\cdot,\cdot)\in C(\Omega;H^{\ell-1}(D))$.
\end{assumption}
Once this input feature $\om\in\Omega$ passes through the deep neural network, the coefficients $\{\widehat{\alpha}_k\}$ are generated as an output. We then reconstruct the solution by
\begin{equation}\label{sol_recon}
\widehat{u}_h(x,\om)=\sum_{k=1}^{N_h}\widehat{\alpha}_k(\om)\phi_k(x).
\end{equation}
For the training of the neural network, we set the population loss function as the $L^2$-residual of the variational formulation \eqref{var_for}: 
\begin{equation}\label{pop_loss}
    \LL(\alpha)=\mathbb{E}_{\om\sim\mathbb{P}_{\Omega}}\bigg[\sum^{N_h}_{k=1}|B[\widehat{u}_h(x,\om),\phi_k(x)]-\ell(\phi_k(x))  |^2\bigg]^{\frac{1}{2}}.
\end{equation}
For the computational efficiency, when we actually compute the approximate solution, we deal with the empirical loss function which is the Monte--Carlo integration of \eqref{pop_loss}:
\begin{equation}\label{emp_loss}
    \LL^M(\alpha)=\frac{|\Omega|}{M}\sum^M_{j=1}\bigg[\sum^{N_h}_{k=1}|B[\widehat{u}_h(x,\om_j),\phi_k(x)]-\ell(\phi_k(x))|^2\bigg]^{\frac{1}{2}},
\end{equation}
where $\{\om_j\}_{j=1}^{M}$ is a family of i.i.d. random samples selected from $\mathbb{P}_{\Omega}$. For each training epoch, the parameters of the neural networks are updated toward minimizing the empirical loss $\LL^M$, and then the given forcing term goes through the neural network again to predict more accurate coefficients. This process is repeated until we achieve a sufficiently small loss, and then we finally obtain the solution prediction \eqref{sol_recon}. It is noteworthy that the FEONet only uses sample parameters randomly selected from $\Omega$ for the training, and hence it can be trained without any precomputed pairs of input-output training data. Moreover, as our method is based on the basis expansion \eqref{sol_recon}, one can impose exact boundary values to the numerical solutions as we can do in the classical FEM. A schematic illustration of the decomposition of the FEONet is presented in Figure \ref{figure:scheme}. Various numerical experiments on several benchmark problems can be found in \cite{FEONet}, where we can confirm that our approach exhibited excellent performance, proving its versatility in terms of accuracy, generalization, and computational flexibility.
\begin{figure}[t]
\begin{center}
\includegraphics[width=0.8\textwidth ]{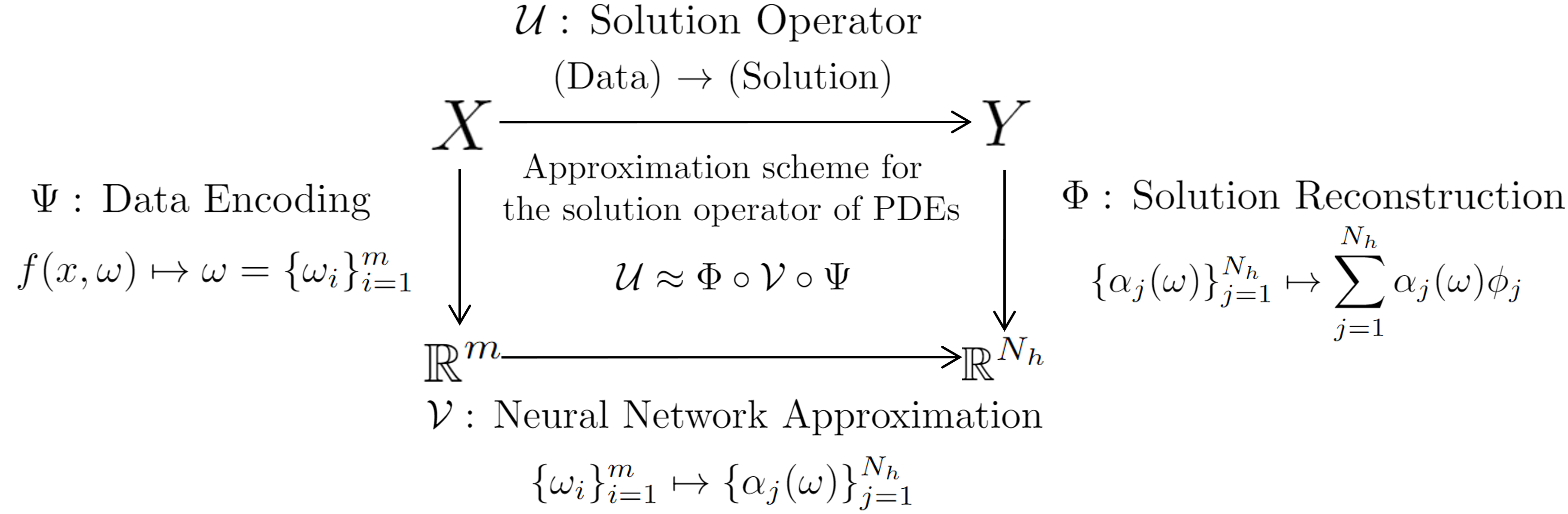}
\caption{Schematic illustration of the decomposition of the FEONet into the encoder $\Psi$, approximator $\mathcal{N}$ and reconstructor $\Phi$.}
\label{figure:scheme}
\end{center}
\end{figure}

\subsection{Estimations for eigenvalues of finite element matrices}\label{subsec:eigen_est}
As our FEONet approach is based on the finite element method, it is reasonable to expect that similar mathematical issues for the FEM may arise in the analysis of the FEONet. In fact, it will be made clear in the later analysis that the convergence of the FEONet is related to the condition number of the finite element matrices. In the classical FEM, adaptive or misshaped meshes may have an impact on the conditioning for the finite element approximation. This has a strong influence on the convergence of the approximation and the performance of iterative solvers for the resulting linear algebraic systems. For this reason, a condition number of finite element matrices has had a long history and has been of particular interest both theoretically and practically. Regarding this topic, a number of estimates of the largest and the smallest eigenvalues are available for both isotropic and anisotropic mesh geometries, and some conditioning techniques have been established in various directions
\cite{BS_book, eigen_3, eigen_7, eigen_4, small_eigen_1}.
In the present paper, the discussion on such matters in full generality is out of scope, and we will only focus on the case of shape-regular partitions as we defined in the previous section. Therefore, the results we shall present here are not new and they can be regarded as special cases of the aforementioned results. However, for the sake of completeness, we will briefly present the derivation of the estimates in this section. These results will be used in the convergence analysis of the FEONet in later sections.

Let $A=S+C$ where the matrices $S$ and $C$ are defined in \eqref{vec_form}. For the estimation of the largest eigenvalue, we essentially follow the idea presented in \cite{BS_book}. Note here that since we are considering the shape-regular meshes, we may identify for all $E\in\mathcal{G}_h$ that $h_E\simeq h$, and $h^d\simeq N_h^{-1}$. By the inverse inequality (e.g. Eq. (1.5) in \cite{BS_book}), we have
\begin{align*}
\alpha^TA\alpha&=B[u_h,u_h]\lesssim\|u_h\|^2_{H^1(D)}=\sum_{E\in\mathcal{G}_h}\|u_h\|^2_{H^1(E)}\lesssim\sum_{E\in\mathcal{G}_h}h^{d-2}_E\|u_h\|^2_{L^{\infty}(E)}\\
&\lesssim h^{d-2}\sum_{E\in\mathcal{G}_h}\sum_{{\rm{supp}}(\phi_i)\cap E\neq\emptyset }\alpha^2_i\lesssim h^{d-2}\alpha^T\alpha.
\end{align*}
This implies that 
\begin{equation}\label{max_eigen_1}
\lambda_{\max}\lesssim h^{d-2}\simeq N_h^{-1+\frac{2}{d}},    
\end{equation}
where $\lambda_{\max}$ denotes the largest eigenvalue of $A$. Next, we shall prove that $N_h^{-1}$ is the lower bound of the smallest eigenvalue. By Poincar\'e's inequality and the inverse inequality, we have
\begin{align*}
    \alpha^TA\alpha
    &=B[u_h,u_h]\gtrsim \|u_h\|^2_{H^1(D)}\gtrsim\|u_h\|^2_{L^2(D)}=\sum_{E\in\mathcal{G}_h}\|u_h\|^2_{L^2(E)}\\
    &\gtrsim h^d\sum_{E\in\mathcal{G}_h}\|u_h\|^2_{L^{\infty}(E)}\gtrsim h^d \sum_{E\in\mathcal{G}_h}\sum_{{\rm{supp}}(\phi_i)\cap E\neq\emptyset}\alpha^2_i\gtrsim h^d\alpha^T\alpha.
\end{align*}
Therefore, we obtain
\begin{equation}\label{min_eigen}
    \lambda_{\min}\gtrsim h^d\simeq N_h^{-1},   
\end{equation}
where $\lambda_{\min}$ denotes the smallest eigenvalue of $A$.

Now let us make some comments on the estimate for the condition number of $A$. If the $\ell^2$ norm is used to define the condition number, we can see that
\begin{equation}\label{cond_sing}
    \kappa(A)=\frac{\sigma_{\max}}{\sigma_{\min}},
\end{equation}
where $\sigma_{\max}$ and $\sigma_{\min}$ are maximal and minimal singular values of $A$ respectively. Furthermore, if the matrix $A$ is normal (i.e., it commutes its conjugate transpose), it is known that (see, e.g., \cite{mat_anal}) the condition number $\kappa(A)$ of $A$ satisfies
\begin{equation}\label{cond_for}
\kappa(A)=\frac{|\lambda_{\max}|}{|\lambda_{\min|}}.
\end{equation}
Therefore, if $A$ is normal, together with the estimates \eqref{max_eigen_1} and \eqref{min_eigen}, we may conclude that
\begin{equation}\label{cond_est}
\kappa(A)\lesssim N_h^{\frac{2}{d}}\simeq h^{-2},    
\end{equation}
where $N_h$ is the degree of freedom in the finite element formulation. On the other hand, in \cite{cond_lower}, the result for the lower bound of the condition number was obtained. More precisely, the authors proved the following estimate owing to the shape-regularity of the triangulation $\{\mathcal{G}_h\}_{h>0}$:
\begin{equation}\label{cond_lower_bound}
    h^{-2}\lesssim \kappa(A).
\end{equation}
From \eqref{cond_lower_bound}, \eqref{max_eigen_1} and \eqref{min_eigen}, we can derive
\begin{equation}\label{inv_min_eigen}
    \lambda_{\max}\lesssim\kappa(A)^{1-d/2}\quad{\rm{and}}\quad{\lambda_{\min}}^{-1}\lesssim \kappa(A)^{d/2},
\end{equation}
which we will use later. The estimates obtained in this section will be utilized in the following sections to verify the relationship between the condition number and the convergence of the FEONet.

\subsection{Feed forward neural networks}\label{subsec:nn}
Next, we shall define a class of feed-forward ReLU neural networks, which we will consider throughout the paper. For each $L\in\mathbb{N}$, we denote an $L$-layer ReLU neural network by a function $f^L(x):\R^{n_0}\rightarrow\R^{n_L}$, defined recursively by
\begin{equation}\label{nn_def}
    f^1(x)=W^1x+b^1 \quad{\rm{and}}\quad
    f^{\ell}(x)=W^{\ell}\sigma(f^{\ell-1}(x))+b^{\ell}\,\,\,{\rm{for}}\,\,2\leq\ell\leq L,
\end{equation}
where $W^{\ell}\in\R^{n_{\ell}\times n_{\ell-1}}$ and $b^{\ell}\in\R^{n_{\ell}}$ is the weight matrix and the bias vector respectively for the $\ell$-th layer. Here, $\sigma:\R\rightarrow\R$ is the ReLU activation function defined by $\sigma(x)=\max\{0,x\}$, and $\sigma(x)$ signifies the vector $(\sigma(x_1),\cdots,\sigma(x_k))$ for $x=(x_1,\cdots,x_k)$. We denote the architecture of a given network by the vector ${\vec{\boldsymbol{n}}}=(n_0,\cdots,n_L)$, the family of neural network parameters by $\theta:=\theta_{\vec{\boldsymbol{n}}}=\{(W^1,b^1),\cdots,(W^L,b^L)\}$, and its realization as a function by $\mathcal{R}[\theta](x)$. For a given neural network architecture $\vec{\boldsymbol{n}}$, the collection of all possible parameters is defined by
\begin{equation}\label{nn_para_set}
    \Theta_{\vec{\boldsymbol{n}}}=\left\{\{(W^{\ell},b^{\ell})\}^L_{\ell=1}:W^{\ell}\in\R^{n_{\ell}\times n_{\ell-1}},\,b^{\ell}\in\R^{n_{\ell}}\right\}.
\end{equation}
  For two feed-forward neural networks $\theta_i$, $i=1,2$ with architectures  $\vec{\boldsymbol{n}}_i=\left(n^{(i)}_0,\cdots,n^{(i)}_{L_i}\right)$, we shall write $\vec{\boldsymbol{n}}_1\subset\vec{\boldsymbol{n}}_2$ if for any $\theta_1\in\Theta_{\vec{\boldsymbol{n}}_1}$, there exists $\theta_2\in\Theta_{\vec{\boldsymbol{n}}_2}$ satisfying $\mathcal{R}[\theta_1](x)=\mathcal{R}[\theta_2](x)$ for all $x\in\R^{n_0}$.
 
 Now let us assume that there exists a sequence of neural network architectures $\{\vec{\boldsymbol{n}}_n\}_{n\geq1}$ satisfying $\vec{\boldsymbol{n}}_n\subset \vec{\boldsymbol{n}}_{n+1}$ for all $n\in\mathbb{N}$. We define the corresponding family of neural networks by
\begin{equation}\label{nn_class}
    \NN_n=\{\mathcal{R}[\theta]:\theta\in\Theta_{\vec{\boldsymbol{n}}_n}\}.
\end{equation}
It is straightforward to verify that $\NN_n\subset\NN_{n+1}$ for any $n\in\mathbb{N}$. We shall exploit the following theorem in the later analysis.
\begin{theorem}\label{ass_1}
Let $K$ be a compact set in $\R^{m}$ and assume that $g\in C(K,\R^{N})$. Then there holds
\begin{equation}\label{UAT}
    \lim_{n\rightarrow\infty}\inf_{\hat{g}\in\NN_n}\|\hat{g}-g\|_{C(K)}=0.
\end{equation}
\end{theorem}
It is called the universal approximation theorem and is known to hold for various scenarios. For instance, the original form of the universal approximation theorem for two-layer neural networks ($L=2$) with an arbitrary number of nodes $n$ exactly coincides with the above theorem \cite{UA_1, UA_2}. On the other hand, in the recent paper \cite{terry_lion}, the authors addressed an extension on the networks of arbitrary depth with bounded width. The authors assumed that the activation function $\sigma$ is non-affine and continuously differentiable at some points, with non-vanishing derivatives at these points. For this case, the authors considered the class of neural networks with an arbitrary number of layers with $m+N+2$ neurons for each, and derived \eqref{UAT}.

\subsection{Analytic framework}
In this section, we shall describe the analytic framework which we will work within. Throughout the paper, given that the external force $f(x,\om)$ parametrized by $\om\in\Omega$, $u(x,\om)$ signifies the corresponding solution, and $u_h(x,\om)$ denotes the finite element approximation 
\begin{equation}\label{for_0}
    u_h(x,\om)=\sum_{k=1}^{N_h}\alpha^*_k(\om)\phi_k(x),    
\end{equation}
where $\alpha^*(\om)=\{\alpha^*_k(\om)\}_{k=1}^{N_h}$ is the set of target finite element coefficients obtained from \eqref{LAS}.

\begin{comment}
Note that both the solution $u$ and the approximation $u_h$ depend on the parametric variable $\om\in\Omega$ and $\alpha^*(\om)$ can be interpreted as a function contained in $C(\Omega,\R^{N_h})$ provided that $f(x,\om)$ satisfies \eqref{f_ass}. In this perspective, a key observation is that $\alpha^*$ is the solution to the minimization problem which is equivalent to \eqref{gal_approx}:
\begin{equation}\label{for_1}
    \alpha^*=\argmin_{\alpha\in C(\Omega,\R^{N_h})} \LL(\alpha),
\end{equation}
where $\LL$ is defined in \eqref{pop_loss}.
\end{comment}

Next, we consider the approximation of the coefficients $\alpha^*$ by $\R^{N_h}$-valued ReLU neural networks. We seek for $\widehat{\alpha}^{\LL}_n:\Omega\rightarrow\R^{N_h}$ solving the minimization problem
\begin{equation}\label{for_2}
    \widehat{\alpha}^{\LL}_n=\argmin_{\alpha\in\NN_n}\LL(\alpha),
\end{equation}
where the minimization is over the neural networks class $\NN_n$. Then we denote the associated solution by
\begin{equation}\label{for_3}
    u_{h,n}(x,\om)=\sum^{N_h}_{k=1}(\widehat{\alpha}^{\LL}_{n})_k(\om)\phi_k(x).
\end{equation}

Finally, we shall consider the solution for the discrete minimization over the family of neural networks
\begin{equation}\label{for_4}
    \widehat{\alpha}^{\LL}_{n,M}=\argmin_{\alpha\in\NN_n}\LL^M(\alpha),
\end{equation}
where the empirical loss $\LL^M$ was defined in \eqref{emp_loss}. In addition, we write the corresponding solution as
\begin{equation}\label{for_5}
    u_{h,n,M}(x,\om)=\sum^{N_h}_{k=1}(\widehat{\alpha}^{\LL}_{n,M})_k(\om)\phi_k(x).
\end{equation}

In the present paper, we shall assume that the optimization error is negligible and that we can always find the exact minimizer for the optimization problems \eqref{for_2} and \eqref{for_4}. Therefore, we shall regard $\widehat{\alpha}_{n,M}^{\LL}$ as the neural network approximation for the target coefficients $\alpha^*$, and $u_{h,n,M}$ is our solution prediction computed by our FEONet scheme. 

The main objective of this paper is to investigate the error $\|u-u_{h,n,M}\|_{L^1(\Omega;L^2(D))}$. For this purpose, we split it into three parts:
\begin{equation}\label{error_split}
    u-u_{h,n,M}=
    \underbrace{(u-u_{h})}_{\text{FEM error}}+
    \underbrace{(u_{h}-u_{h,n})}_{\text{ 
 approximation error}}+
    \underbrace{(u_{h,n}-u_{h,n,M})}_{\text{
 generalization error}}
    =:{\rm{(I)}}+{\rm{(II)}}+{\rm{(III)}}.
\end{equation}
The first error ${\rm{(I)}}$ arises from the finite element approximation and its mathematical analysis is fairly well-known. In particular, under suitable assumptions, it is known that ${\rm{(I)}}$ converges to $0$ as $h\rightarrow0$, and the convergence rate improves as we approximate the solution with higher-order polynomials assuming that the solution has suitable regularity properties. The second error ${\rm{(II)}}$ is referred to as the approximation error, occurring when the true coefficient $\alpha^*$ in \eqref{for_0} is approximated by a deep neural network. The final error ${\rm{(III)}}$ is known as the generalization error, which measures how well our predicted solution $u_{h,n,M}$ trained with $M$ random samples generalizes for other samples that were not used in the training process; in other words, it is the error mainly caused by the approximation of $\LL$ by $\LL^M$.

In the previous works \cite{ULG_anal, FEONet} where similar schemes were analyzed, all the convergence results were studied only when the index for the baseline numerical method (here denoted by $h>0$) was fixed. Furthermore, the effect of $h>0$ (or $N\in\mathbb{N}$ in \cite{ULG_anal} where the Legendre--Galerkin method was considered) to the errors ${\rm{(II)}}$ and ${\rm{(III)}}$ was not investigated in the previous analysis. In the present paper, however, we will derive a comprehensive error estimate, simultaneously examining the errors ${\rm{(I)}}$, ${\rm{(II)}}$, and ${\rm{(III)}}$ without fixing $h>0$, and analyzing the role of the indices $h>0$, $n$, $M\in\mathbb{N}$ in each error in \eqref{error_split}.

As will be made clear in the later analysis, the convergence of the errors ${\rm{(II)}}$ and ${\rm{(III)}}$ is closely related to the condition number of the finite element matrices defined in \eqref{vec_form}. Based on the condition number estimates derived in Section \ref{subsec:eigen_est}, we will rigorously investigate the role of $h>0$ on the approximation and generalization errors, which is a key novel part of the paper compared to the previous work \cite{ULG_anal, FEONet}. 
\begin{comment}
As a first step, we will conduct the convergence analysis for the corresponding neural network coefficients; in other words, we shall prove the inequality \begin{equation}\label{nn_error}
    \|\alpha^*-\widehat{\alpha}^{\LL}_n\|_{L^1(\Omega)} + \|\widehat{\alpha}^{\LL}_n-\widehat{\alpha}^{\LL}_{n,M}\|_{L^1(\Omega)}\lesssim C(h,n,M).
\end{equation}
These error estimates for the neural network coefficients are crucial in our analysis and will be discussed separately in the next section.
\end{comment}

\section{Convergence analysis for approximate solutions}\label{sec:conv}
We start with the following observation on the loss functions defined in \eqref{pop_loss} and \eqref{emp_loss}. By the definition of the loss functions, the solution representation \eqref{sol_recon} and the bilinearity of $B[\cdot,\cdot]$, we note that
\begin{equation}\label{obs_J}
\begin{aligned}
    \LL(\alpha)
    \begin{comment}
    &=\int_{\Omega}\left[\sum^{N_h}_{i=1}\bigg|B\bigg[\sum^{N_h}_{k=1}\alpha_k(\om)\phi_k(x),\phi_k(x)\bigg]-l(\phi_i(x);\om)\bigg|^2\right]^{\frac{1}{2}}\dom\\
    \end{comment}
    &=\int_{\Omega}\left[\sum^{N_h}_{i=1}\bigg|\sum^{N_h}_{k=1}\alpha_k(\om)B[\phi_k(x),\phi_i(x)]-\ell(\phi_i(x)
    )\bigg|^2\right]^{\frac{1}{2}}\dom\\
    &=\int_{\Omega}\left[\sum^{N_h}_{i=1}\bigg|(A\alpha(\om))_i-(F(\om))_i\bigg|^2\right]^{\frac{1}{2}}\dom=\norm{A\alpha(\om)-F(\om)}_{L^1(\Omega)},
\end{aligned}
\end{equation}
where $A$ and $F$ are defined in \eqref{vec_form}. In the same manner, by the definition of population loss, we have
\begin{equation}\label{obs_JM}
    \LL^M(\alpha)=\frac{|\Omega|}{M}\sum_{j=1}^M|A\alpha(\om_j)-F(\om_j)|,
\end{equation}
where each $\om_j$ is randomly chosen according to the distribution of $\mathbb{P}_{\Omega}$. With the aid of the above observation, the loss functions $\LL$ and $\LL^M$ are represented in terms of the finite element matrices, and the analysis has now become a matter of the properties of these matrices. It is noteworthy that the matrix $A$ contains the information of the given PDE and boundary conditions, and thus the characterization 
of $A$ which can cover various PDE settings are of importance. In this perspective, the following lemma is useful in the analysis of the FEONet scheme regarding the structure of $A$, which is a direct consequence of the spectral theorem (see, e.g., Proposition 3.1 in \cite{ULG_anal}).
\begin{lemma}\label{matrix_thm}
Suppose that $T$ is a $N\times N$ symmetric and positive-definite matrix and let us denote minimum and maximum eigenvalues of $T$ by $\lambda_{\min}$ and $\lambda_{\max}$ respectively. Then for any $x\in\R^N$, we have
\begin{equation}\label{eigen_est}
    \lambda_{\min}|x|\leq|Tx|\leq\lambda_{\max}|x|.
\end{equation}
\end{lemma}
In this section, we shall use the above lemma with the matrix $T=A^TA$, which is symmetric. Furthermore, by coercivity of $B[\cdot,\cdot]$, 
\begin{comment}
we have for any nonzero vector $y\in\R^{N_h}$,
\begin{equation}\label{coerc_PD}
    y\cdot Ay=\sum^{N_h}_{i=1}\sum^{N_h}_{j=1}y_iy_jB[\phi_i,\phi_j]=B\bigg[\sum^{N_h}_{i=1}y_i\phi_i,\sum^{N_h}_{j=1}y_j\phi_j\bigg]\gtrsim\bigg\|\sum^{N_h}_{j=1}y_j\phi_j\bigg\|^2_{H^1}>0,
\end{equation}
\end{comment}
we find that $A$ is positive-definite, and hence, $A^TA$ is also positive-definite. Therefore, even though $A$ is not symmetric, we can apply Lemma \ref{matrix_thm} with $T=A^TA$. Based on this fact, in this section, we will derive estimates for the approximation error and the generalization error for the approximate coefficients $\|\alpha^*-\widehat{\alpha}^{\LL}_n\|_{L^1(\Omega)}$ and $\|\widehat{\alpha}^{\LL}_n-\widehat{\alpha}^{\LL}_{n,M}\|_{L^1(\Omega)}$ respectively. Before proceeding further, we introduce modified loss functions defined by
\[
    \HH(\alpha)=\|A^TA\alpha(\om)-A^TF(\om)\|_{L^1(\Omega)}\quad{\rm{and}}\quad \HH^M(\alpha)=\frac{|\Omega|}{M}\sum_{i=1}^M|A^TA\alpha(\om_i)-A^T F(\om_i
    )|,
\]
and their minimizers over the class of neural networks
\[
    \widehat{\alpha}^{\HH}_n=\argmin_{\alpha\in\NN_n}\HH(\alpha)\quad{\rm{and}}\quad \widehat{\alpha}^{\HH}_{n,M}=\argmin_{\alpha\in\NN_n}\HH^M(\alpha).
\]

\subsection{Approximation error and Ce\'a's lemma}
In this section, we aim to analyze the approximation error for the approximate coefficients. By Assumption \ref{f_ass}, it is easy to verify that $F\in C(\Omega;\R^{N_h})$ and hence $\alpha^*=A^{-1}F\in C(\Omega;\R^{N_h})$. The precise statement for the approximation error estimate is encapsulated in the following theorem, which is a version of Ce\'a's lemma for the neural network approximation of the FEONet.
\begin{theorem}[Approximation error]\label{approx_error}
    Suppose that Assumption \ref{f_ass} holds and let $\kappa(A)$ denote the condition number of the finite element matrix $A=S+C$ defined in \eqref{vec_form}. For the finite element coefficients $\alpha^*\in C(\Omega;\R^{N_h})$ and the approximate coefficient $\widehat{\alpha}^{\LL}_n\in\NN_n$, we have
    \begin{equation}\label{approx_est}
        \|\alpha^*-\widehat{\alpha}^{\LL}_n\|_{L^1(\Omega)}\lesssim \kappa(A)^2\inf_{\alpha\in\NN_n}\|\alpha-\alpha^*\|_{L^1(\Omega)}.
    \end{equation}
\end{theorem}
\begin{proof}
By the triangle inequality, we can see that
\begin{equation}\label{split_approx_error}
    \|\alpha^*-\widehat{\alpha}^{\LL}_n\|_{L^1(\Omega)}\leq\|\alpha^*-\widehat{\alpha}^{\HH}_n\|_{L^1(\Omega)}+\|\widehat{\alpha}^{\HH}_n-\widehat{\alpha}^{\LL}_n\|_{L^1(\Omega)}.
\end{equation}
If we let $\rho_{\max}$ and $\rho_{\min}$ be maximum and minimum singular values of $A$ respectively, (i.e., the square root of the eigenvalue of $A^TA$), from the repetitive application of Lemma \ref{matrix_thm} with $T=A^TA$, we have
\begin{align*}
    \|\alpha^*-\widehat{\alpha}^{\HH}_n\|_{L^1(\Omega)}
    %&\leq\frac{1}{(\rho_{\rm{min}})^2}\|A^TA\alpha^*-A^TA\widehat{\alpha}^{\HH}_n\|_{L^1(\Omega)}\\
    &\leq\frac{1}{(\rho_{\rm{min}})^2}\left(\|A^TA\alpha^*-A^TF\|_{L^1(\Omega)}+\|A^TA\widehat{\alpha}^{\HH}_n-A^TF\|_{L^1(\Omega)}\right)\\
    &=\frac{1}{(\rho_{\rm{min}})^2}\HH(\widehat{\alpha}^{\HH}_n)\leq\frac{1}{(\rho_{\rm{min}})^2}\inf_{\alpha\in\NN_n}\HH(\alpha)=\frac{1}{(\rho_{\rm{min}})^2}\inf_{\alpha\in\NN_n}\|A^TA\alpha-A^TF\|_{L^1(\Omega)}\\
    &\leq\frac{1}{(\rho_{\rm{min}})^2}\inf_{\alpha\in\NN_n}\left(\|A^TA\alpha-A^TA\alpha^*\|_{L^1(\Omega)}+\|A^TA\alpha^*-A^TF\|_{L^1(\Omega)}\right)\\
    &\leq\left(\frac{\rho_{\rm{max}}}{\rho_{\rm{min}}}\right)^2\inf_{\alpha\in\NN_n}\|\alpha-\alpha^*\|_{L^1(\Omega)}
    =\kappa(A)^2\inf_{\alpha\in\NN_n}\|\alpha-\alpha^*\|_{L^1(\Omega)},
\end{align*}
where we have used the minimality of $\widehat{\alpha}^{\HH}_n$ and the fact that $A^TA\alpha^*=A^TF$. 

Next, let us estimate the second term on the right-hand side of \eqref{split_approx_error}. If we write the largest eigenvalue of $A$ as $\lambda_{\max}$, by using the facts that $\|A^T\|=\|A\|=\rho_{\max}$ and $|\lambda_{\max}|\leq\rho_{\max}$ (see, e.g., \cite{mat_anal}), we have
\begin{align*}
    \|\widehat{\alpha}^{\HH}_n-\widehat{\alpha}^{\LL}_n\|_{L^1(\Omega)}
    %&\leq\frac{1}{(\rho_{\min})^2}\|A^TA\widehat{\alpha}^{\HH}_n-A^TA\widehat{\alpha}^{\LL}_n\|_{L^1(\Omega)}\\
    &\leq\frac{1}{(\rho_{\min})^2}\left(\|A^TA\widehat{\alpha}^{\HH}_n-A^TF\|_{L^1(\Omega)}+\|A^TA\widehat{\alpha}^{\LL}_n-A^TF\|_{L^1(\Omega)}\right)\\
    &\leq\left(\frac{\rho_{\max}}{\rho_{\min}}\right)^2\inf_{\alpha\in\NN_n}\|\alpha-\alpha^*\|_{L^1(\Omega)}+\frac{\|A^T\|}{(\rho_{\min})^2}\|A\widehat{\alpha}^{\LL}_n-F\|_{L^1(\Omega)}\\
    &\leq\left(\frac{\rho_{\max}}{\rho_{\min}}\right)^2\inf_{\alpha\in\NN_n}\|\alpha-\alpha^*\|_{L^1(\Omega)}+\frac{\rho_{\max}}{(\rho_{\min})^2}\lambda_{\max}\inf_{\alpha\in\NN_n}\|\alpha-\alpha^*\|_{L^1(\Omega)}\\
    &\leq2\left(\frac{\rho_{\max}}{\rho_{\min}}\right)^2\inf_{\alpha\in\NN_n}\|\alpha-\alpha^*\|_{L^1(\Omega)}=2\kappa(A)^2\inf_{\alpha\in\NN_n}\|\alpha-\alpha^*\|_{L^1(\Omega)}.
\end{align*}
\end{proof}

One interesting point is that the convergence of the approximate coefficient depends on the condition number of $A$. The FEONet uses a basis expansion similar to the classical FEM, but the essential difference is that the FEM solves the algebraic equation $A\alpha=F$, while the FEONet seeks to find a minimizer of $|A\alpha-F|$ over the class of neural networks. In the traditional FEM, iterative methods are commonly used to solve the equation $A\alpha=F$ (e.g., Krylov subspace methods, conjugate gradient methods) and the convergence of iterative methods depends on the condition number, which eventually affects the overall computation of the solution in the FEM. Therefore, it is natural to suspect whether the condition number also plays a role in computing the solution prediction of the FEONet, and the above theorem quantitatively confirms this relationship. As in the traditional FEM, we have seen that a high condition number also hinders the computation of solutions by the FEONet.

Another intriguing point is that, for a fixed finite element parameter $h>0$, (and hence the condition number $\kappa(A)$), by the universal approximation theorem \eqref{UAT},  we can prove that
\begin{equation}\label{app_t_est}
\|\alpha^*-\widehat{\alpha}^{\LL}_n\|_{L^1(\Omega)}\rightarrow0\quad{\rm{as}}\,\,n\rightarrow \infty,
\end{equation}
which is the desired convergence property of our approximate coefficients.

\subsection{Generalization error}
For the generalization error, we begin with the following definition so-called {\textit{Rademacher complexity}}.

\begin{definition}
Let $\{X_i\}_{i=1}^M$ be a sequence of i.i.d. random variables. For a function class $\mathcal{F}$, we define the Rademacher complexity  by
\[
    \mathcal{R}_M(\mathcal{F})=\mathbb{E}_{\{X_i\}^M_{i=1}}\mathbb{E}_{\{\varepsilon_i\}^M_{i=1}}\bigg[\sup_{f\in\mathcal{F}}\bigg|\frac{1}{M}\sum^M_{i=1}\varepsilon_if(X_i)\bigg|\bigg],
\]
where $\varepsilon_i$'s denote i.i.d. Bernoulli random variables, which means that $\mathbb{P}(\varepsilon_i=1)=\mathbb{P}(\varepsilon_i=-1)=\frac{1}{2}$.
\end{definition}
As we can see from the above definition, the Rademacher complexity is the expectation value of the maximum correlation between the vector $(f(X_1),\cdots,f(X_M))$ and the random noise $(\varepsilon_1,\cdots,\varepsilon_M)$, where the maximum is taken over the family of functions $\mathcal{F}$. This measures the capability of the class $\mathcal{F}$ to fit random noise. For comprehensive information on the Rademacher complexity, see \cite{Rade_1, Rade_2}. The next theorem states that the difference between a population loss and an empirical loss can be bounded by the Rademacher complexity (see, for instance, Proposition 4.11 in \cite{Rade_2}).

\begin{theorem}\label{Rade_thm}
    Let $\mathcal{F}$ be a class of function and $\{X_i\}_{i=1}^M$ be a sequence of i.i.d. random variables. Then we have the following inequality
    \begin{equation}\label{Rade_ineq}
        \mathbb{E}\bigg[\sup_{f\in\mathcal{F}}\bigg|\frac{1}{M}\sum_{j=1}^Mf(X_j)-\mathbb{E}_{X\sim\mathbb{P}_{\Omega}}f(X)\bigg|\bigg]\leq 2\mathcal{R}_M(\mathcal{F}),
    \end{equation}
    where the expectation is taken for the random variables $\{X_i\}_{i=1}^M$.
\end{theorem}

Next, we define the following function class of interest concerning the loss functions:
\begin{equation}\label{fct_class}
    \mathcal{F}^{\LL}_n:=\{|A\alpha-F|:\alpha\in\NN_n\},
    \quad \mathcal{F}^{\HH}_n:=\{|A^TA\alpha-A^TF|:\alpha\in\NN_n\}.
\end{equation}
Now, we shall utilize Theorem \ref{Rade_thm} with both $\mathcal{F}=\mathcal{F}^{\LL}_n$ and $\mathcal{F}=\mathcal{F}^{\HH}_n$ to derive the estimate for the generalization error. The precise statement is presented in the following theorem.

\begin{theorem}[Generalization error]\label{coef_gen_err}
    Let Assumption \ref{f_ass} holds and $\kappa(A)$ be the condition number of the finite element matrix $A=S+C$ defined in \eqref{vec_form}. Then we have
    \begin{equation}\label{general_est}
        \EE\left[\|\widehat{\alpha}^{\LL}_n-\widehat{\alpha}^{\LL}_{n,M}\|_{L^1(\Omega)}\right]\lesssim\kappa(A)^{1+d/2}\mathcal{R}_M(\mathcal{F}^{\LL}_n)+\kappa(A)^{d+2}{R}_M(\mathcal{F}^{\HH}_n)+\kappa(A)^2\inf_{\alpha\in\NN_n}\|\alpha-\alpha^*\|_{L^1(\Omega)}.
    \end{equation}
\end{theorem}
\begin{proof}
    By the triangle inequality, we have
    \begin{align*}    
    \EE\left[\|\widehat{\alpha}^{\LL}_n-\widehat{\alpha}^{\LL}_{n,M}\|_{L^1(\Omega)}\right]
    &\leq \EE\left[\|\widehat{\alpha}^{\LL}_n - \widehat{\alpha}^{\HH}_n  \|_{L^1(\Omega)}\right]+\EE\left[\|\widehat{\alpha}^{\HH}_n - \widehat{\alpha}^{\HH}_{n,M}  \|_{L^1(\Omega)}\right]+\EE\left[\|\widehat{\alpha}^{\HH}_{n,M} - \widehat{\alpha}^{\LL}_{n,M}  \|_{L^1(\Omega)}\right]\\
    &=:{\rm{(I)}}+{\rm{(II)}}+{\rm{(III)}}
    \end{align*}
    From the argument used in the proof of Theorem \ref{approx_error}, we know that
    \begin{equation}
        {\rm{(I)}}\lesssim \kappa(A)^2\inf_{\alpha\in\NN_n}\|\alpha-\alpha^*\|_{L^1(\Omega)}.
    \end{equation}
Next, through the repeated applications of Lemma \ref{matrix_thm} and Theorem \ref{Rade_thm}, we obtain
    \begin{align*}
        {\rm{(II)}}
        %&\leq\frac{1}{\rho_{\min}^2}\EE\left[\|A^TA\widehat{\alpha}^{\HH}_n-A^TA\widehat{\alpha}^{\HH}_{n,M}\|_{L^1(\Omega)}\right]\\
        &\leq\frac{1}{\rho^2_{\min}}\EE\left[\|A^TA\widehat{\alpha}^{\HH}_n-A^TF\|_{L^1(\Omega)}+\|A^TA\widehat{\alpha}^{\HH}_{n,M}-A^TF\|_{L^1(\Omega)}\right]\\
        &=\frac{1}{\rho^2_{\min}}\EE\left[\HH(\widehat{\alpha}^{\HH}_n)+\HH(\widehat{\alpha}^{\HH}_{n,M})\right]\lesssim\frac{1}{\rho^2_{\min}}\EE\left[\HH(\widehat{\alpha}^{\HH}_{n,M})\right]\\
        &\lesssim\frac{1}{\rho^2_{\min}}\left(\EE\left[\HH(\widehat{\alpha}^{\HH}_{n,M})-\HH^M(\widehat{\alpha}^{\HH}_{n,M})\right]+\EE\left[\HH^M(\widehat{\alpha}^{\HH}_{n})\right]\right)\\
        &\lesssim\frac{1}{\rho^2_{\min}}\RR_M(\mathcal{F}^{\HH}_n)+\frac{2}{\rho^2_{\min}}\EE\left[\HH^M(\widehat{\alpha}^{\HH}_n)-\HH(\widehat{\alpha}^{\HH}_n)\right]+\frac{2}{\rho^2_{\min}}\EE\left[\HH(\widehat{\alpha}^{\HH}_n)\right]\\
        &\lesssim\frac{1}{\rho^2_{\min}}\RR_M(\mathcal{F}^{\HH}_n)+\left(\frac{\rho_{\max}}{\rho_{\min}}\right)^2\inf_{\alpha\in\NN_n}\|\alpha-\alpha^*\|_{L^1(\Omega)},
    \end{align*}
    where we have used the argument used in the proof of Theorem \ref{approx_error} to obtain the last inequality. Finally, 
\begin{align*}
    {\rm{(III)}}
    %&\leq\frac{1}{(\rho_{\min})^2}\EE\left[\|A^TA\widehat{\alpha}^{\HH}_{n,M}-A^TA\widehat{\alpha}^{\LL}_{n,M}\|_{L^1(\Omega)}\right]\\
    &\leq\frac{1}{(\rho_{\min})^2}\left(\EE\left[\|A^TA\widehat{\alpha}^{\HH}_{n,M}-A^TF\|_{L^1(\Omega)}\right]+\EE\left[\|A^TA\widehat{\alpha}^{\LL}_{n,M}-A^TF\|_{L^1(\Omega)}\right]\right)=:\frac{1}{(\rho_{\min})^2}\left[({\rm{I}})'+({\rm{II}})'\right].
\end{align*}
For the remaining terms, note that 
\begin{align*}
    ({\rm{I}})'
    &\leq\EE\left[\HH(\widehat{\alpha}^{\HH}_{n,M})-\HH^M(\widehat{\alpha}^{\HH}_{n,M})\right]+\EE\left[\HH^M(\widehat{\alpha}^{\HH}_{n,M})\right]\lesssim \RR_M(\mathcal{F}^{\HH}_n)+\EE[\HH^M(\widehat{\alpha}^{\HH}_n)]\\
    &\lesssim\RR_M(\mathcal{F}^{\HH}_n)+\EE\left[\HH^M(\widehat{\alpha}^{\HH}_n)-\HH(\widehat{\alpha}^{\HH}_n)\right]+\EE\left[\HH(\widehat{\alpha}^{\HH}_n)\right]\lesssim\RR_M(\mathcal{F}^{\HH}_n)+(\rho_{\max})^2\inf_{\alpha\in\NN_n}\|\alpha-\alpha^*\|_{L^1(\Omega)},
\end{align*}
and subsequently,
\begin{align*}
    ({\rm{II}})'
    &\leq \mathbb{E}\left[\|A^T\|\LL(\widehat{\alpha}^{\LL}_{n,M})\right]\leq\|A^T\|\left(\mathbb{E}\left[\LL(\widehat{\alpha}^{\LL}_{n,M})-\LL^M(\widehat{\alpha}^{\LL}_{n,M})\right]+\mathbb{E}\left[\LL^M(\widehat{\alpha}^{\LL}_{n,M})\right]\right)\\
    &\leq\|A^T\|\left(2\mathcal{R}_M(\mathcal{F}^{\LL}_n)+\mathbb{E}\left[\LL^M(\widehat{\alpha}^{\LL}_n)-\LL(\widehat{\alpha}^{\LL}_n)\right]+\mathbb{E}\left[\LL(\widehat{\alpha}^{\LL}_n)\right]\right)\\
    &\lesssim\|A^T\|\left(4\mathcal{R}_M(\mathcal{F}^{\LL}_n)+\lambda_{\max}\inf_{\alpha\in\NN_n}\|\alpha-\alpha^*\|_{L^1(\Omega)}\right).
\end{align*}
Again, by the facts that $\|A^T\|=\|A\|=\rho_{\max}$ and $|\lambda_{\max}|\leq\rho_{\max}$ together with the above estimates,
\[
({\rm{III}})\lesssim \kappa(A)^{d+2}\mathcal{R}_M(\mathcal{F}_n^{\mathcal{H}})+\kappa(A)^{1+d/2}\mathcal{R}_M(\mathcal{F}^{\LL}_n)+\kappa(A)^2\inf_{\alpha\in\NN_n}\|\alpha-\alpha^*\|_{L^1(\Omega)}.
\]
By combining the above estimates for $({\rm{I}})$, $({\rm{II}})$ and $({\rm{III}})$, together with the fact $|\lambda_{\max}|\leq\rho_{\max}$, we can derive the desired estimate.
\end{proof}
If the finite element parameter $h>0$ is fixed, along with $\rho_{\min}$, $\rho_{\max}$, and $\kappa(A)$, then, according to the universal approximation property, we know that the last term on the right-hand side of \eqref{general_est} converges to zero.
For the first term and the second term, we may assume that the Rademacher complexities converge to zero as $M\rightarrow\infty$; this is a common assumption in statistical learning theory, and it indeed holds for several function families. See, for example, \cite{rade_upper_1, rade_upper_2, rade_upper_3, rade_upper_4}, where this issue was addressed. Therefore, in this case, we can show that for fixed $h>0$, the generalization error converges to zero as $n$, $M\rightarrow\infty.$ More precisely, we have
\begin{equation}\label{gen_err_conv}
    \EE\left[\|\widehat{\alpha}^{\LL}_n-\widehat{\alpha}^{\LL}_{n,M}\|_{L^1(\Omega)}\right]\rightarrow 0\quad{\rm{as}}\,\,\,n,\,M\rightarrow\infty.
\end{equation}

%In the next section, we will investigate certain scenarios where we can explicitly estimate $\rho_{\max}$, $\rho_{\min}$ in terms of $h>0$, and $\inf_{\alpha\in\NN_n}\|\alpha-\alpha^*\|_{L^1(\Omega)}$, $\RR_M(\mathcal{F}^{\LL}_n)$, $\RR_M(\mathcal{F}^{\HH}_n)$ in terms of $n$, $M\in\mathbb{N}$, and hence obtain the complete error bounds for the coefficients approximation. 

\subsection{Convergence of approximate solutions}
Based on the convergence of coefficients \eqref{app_t_est} and \eqref{gen_err_conv}, we shall prove the convergence of the FEONet prediction to the finite element approximation, which is encapsulated in the following theorem.

\begin{theorem}\label{main_thm_whole}
Suppose Assumption \ref{f_ass} holds, and assume that for all $n\in\mathbb{N}$, $\RR_M(\mathcal{F}^{\LL}_n)$ and $\RR_M(\mathcal{F}^{\HH}_n)$ converge to $0$ as $M\rightarrow\infty.$ %For given $h>0$, let us further assume that $\V_h\subset H^s(D)$. and the basis $\{\phi^h_1,\cdots,\phi^h_{N_h}\}$ of $\V_h$ consists of piecewise polynomials of degree less than or equal to $s\in\mathbb{N}$. 
Then we have 
\begin{equation}\label{main_conv_whole}
        \lim_{n\rightarrow \infty}\lim_{M\rightarrow \infty}\mathbb{E}\left[\|u_h-u_{h,n,M}\|_{L^1(\Omega;L^2(D))}\right]=0,
\end{equation}
where the expectation is to take over the random sampling $\om_j\sim\mathbb{P}_{\Omega}$.
\end{theorem}
\begin{proof}
From the definition \eqref{for_0} and \eqref{for_5}, we have that
\begin{comment}
\begin{equation}\label{h_fix_ext}
    \begin{aligned}
       &\|u_h-u_{h,n,M}\|^2_{L^2(\Omega;H^1(D))}\\
    &=\int_{\Omega}\int_D \sum_{|\beta|\leq s} \bigg|\partial^{\beta}\bigg(\sum^{N_h}_{i=1}(\alpha^*_i-(\widehat{\alpha}^{\LL}_{n,M})_i)\phi_i\bigg)\bigg|^2\dx\dom=\int_{\Omega}\int_D \sum_{|\beta|\leq s} \bigg|\sum^{N_h}_{i=1}(\alpha^*_i-(\widehat{\alpha}^{\LL}_{n,M})_i)(\partial^{\beta}\phi_i)\bigg|^2\dx\dom\\
    &=\int_{\Omega}\int_D \sum_{|\beta|\leq s} \bigg|\sum^{N_h}_{i,j=1}(\alpha^*_i-(\widehat{\alpha}^{\LL}_{n,M})_i)(\alpha^*_j-(\widehat{\alpha}^{\LL}_{n,M})_j)\partial^{\beta}\phi_i\partial^{\beta}\phi_j\bigg|\dx\dom\\
    &\lesssim\int_{\Omega}\int_D\sum_{|\beta|\leq s}\bigg(\sum^{N_h}_{i,j=1}|\alpha^*_i-(\widehat{\alpha}^{\LL}_{n,M})_i|^2|\partial^{\beta}\phi_i|^2+\sum^{N_h}_{i,j=1}|\alpha^*_j-(\widehat{\alpha}^{\LL}_{n,M})_j|^2|\partial^{\beta}\phi_j|^2\bigg)\dx\dom\\
    &\lesssim\sum^{N_h}_{i,j=1}\bigg(\int_{\Omega}|\alpha^*_i-(\widehat{\alpha}^{\LL}_{n,M})_i|^2\dom\bigg)\bigg(\sum_{|\beta|\leq s}\int_D|\partial^{\beta}\phi_i|^2\dx\bigg)\\
    &\lesssim C(d,N_h)\sum_{i=1}^{N_h}\bigg(\int_{\Omega}|\alpha^*_i-(\widehat{\alpha}^{\LL}_{n,M})_i|^2\dom\bigg)\|\phi_i\|^2_{H^s(D)}\lesssim C(d,N_h)\max_{1\leq i\leq N_h}\|\phi_i\|^2_{H^s(D)}\|\alpha^*-\widehat{\alpha}^{\LL}_{n,M}\|^2_{L^2(\Omega)}.
    \end{aligned}
\end{equation}
\end{comment}
\begin{equation}\label{h_fix_ext}
\begin{aligned}
    \|u_h-u_{h,n,M}\|_{L^1(\Omega;L^2(D))}
    &=\int_{\Omega}\bigg\|\sum^{N_h}_{j=1}(\alpha^*_j-(\widehat{\alpha}^{\LL}_{n,M})_j)\phi_j\bigg\|_{L^2(D)}\dom\lesssim\int_{\Omega}\bigg(\sum^{N_h}_{j=1}(\alpha^*_j-(\widehat{\alpha}^{\LL}_{n,M})_j)\|\phi_j\|_{L^2(D)}\bigg)\dom\\
    &\lesssim\max_{1\leq i\leq N_h}\|\phi_j\|_{L^2(D)}\int_{\Omega}\bigg(\sum^{N_h}_{j=1}|\alpha^*_j-(\widehat{\alpha}^{\LL}_{n,M})_j|\bigg)\dom\\
    &\lesssim \max_{1\leq i\leq N_h}\|\phi_j\|_{L^2(D)}N_h^{1/2}\|\alpha^*-\widehat{\alpha}^{\LL}_{n,M}\|_{L^1(\Omega)}.
\end{aligned}
\end{equation}
Therefore, for given $h>0$, by \eqref{app_t_est} and \eqref{gen_err_conv}, we conclude that as $n$, $M\rightarrow\infty$,
\[
\mathbb{E}\left[\|u_h-u_{h,n,M}\|_{L^1(\Omega;L^2(D))}\right]\lesssim\mathbb{E}\left[\|\alpha^*-\widehat{\alpha}^{\LL}_{n,M}\|_{L^1(\Omega)}\right]\lesssim \mathbb{E}\left[\|\alpha^*-\widehat{\alpha}^{\LL}_{n}\|_{L^1(\Omega)}\right] + \mathbb{E}\left[\|\widehat{\alpha}^{\LL}_n-\widehat{\alpha}^{\LL}_{n,M}\|_{L^1(\Omega)}\right]\rightarrow 0.
\]
\end{proof}

\section{Error estimates for approximate solutions}\label{sec:err_est}
In the previous section, we considered the general class of second-order elliptic equations \eqref{eq1}-\eqref{eq2} and proved that the error $\|u_h-u_{h,n,M}\|_{L^1(\Omega;L^2(D))}$ goes to zero as $n$, $M\rightarrow\infty$. On the other hand, in this section, we will deal with some particular cases where we can derive an explicit error estimate.

First, we shall consider the case of self-adjoint PDEs, i.e., the case when $\bb(x)=0$ for all $x\in D$. In this case, we can significantly simplify the proof in Section \ref{sec:conv} by avoiding the use of the auxiliary loss function $\HH$. Furthermore, by using the condition number estimates of finite element matrices, we can identify the role of the finite element parameter $h>0$ in the entire convergence and obtain the explicit error bound. Secondly, we will introduce a function space called the Barron space, which is the family of functions endowed with a quantity that can control the approximation and generalization errors by adopting a particular machine
learning model. We will provide a sufficient condition for our target function to be contained in the Barron space, and derive the explicit error bounds which identify the role of the parameters $h>0$, $n$, $M\in\mathbb{N}$ in the overall convergence of the FEONet.

\subsection{Self-adjoint PDEs}
Henceforth, in this section, let us assume that $\bb(x)=0$ for any $x\in D$ in \eqref{eq1}, so that the equations under consideration become self-adjoint:
\begin{align}
    -{\rm{div}}\,(\aaa(x)\nabla u)+c(x)u&=f(x)\quad{\rm{in}}\,\,D,\label{sym_eq1} \\
u(x)&=0\quad\quad\,\,{\rm{on}}\,\,\partial D.\label{sym_eq2}
\end{align}
We shall adjust the analytical setting discussed in Section \ref{sec:prelim} accordingly, so all the notations are now for the equation \eqref{sym_eq1}-\eqref{sym_eq2}. In this case, the finite element matrix $A$ is just $S$ (instead of $S+C$ as before), which is symmetric and positive-definite. Therefore, we can directly apply Lemma \ref{matrix_thm} with $T=A=S$, which allows us to avoid the use of the auxiliary loss function $\HH$. As we can see from the following theorems, since we only work the original loss function $\LL$, the main assertions and the proofs can be significantly simplified. Again, let us denote the condition number of the finite element matrix $A=S$ by $\kappa(A)$, and the largest eigenvalue and the smallest eigenvalue by $\lambda_{\max}$ and $\lambda_{\min}$ respectively. Note that since $A$ is symmetric (and hence normal) and positive-definite, we can write $\kappa(A)=\lambda_{\max}/\lambda_{\min}$.
\begin{theorem}[Approximation error]\label{symm_approx_error}
    Suppose that Assumption \ref{f_ass} holds. For the finite element coefficients $\alpha^*\in C(\Omega;\R^{N_h})$ and the approximate coefficient $\widehat{\alpha}^{\LL}_n\in\NN_n$, we have
    \begin{equation}\label{sym_approx_est}
        \|\alpha^*-\widehat{\alpha}^{\LL}_n\|_{L^1(\Omega)}\leq\kappa(A)\inf_{\alpha\in\NN_n}\|\alpha-\alpha^*\|_{L^1(\Omega)}.
    \end{equation}
\end{theorem}
\begin{proof}
By using Lemma \ref{matrix_thm} with $T=A=S$, it follows that
\begin{align*}
    \|\alpha^*-\widehat{\alpha}^{\LL}_n\|_{L^1(\Omega)}
    %&\leq\frac{1}{\lambda_{\rm{min}}}\|A\alpha^*-A\widehat{\alpha}^{\LL}_n\|_{L^1(\Omega)}
    &\leq\frac{1}{\lambda_{\rm{min}}}\left(\|A\alpha^*-F\|_{L^1(\Omega)}+\|A\widehat{\alpha}^{\LL}_n-F\|_{L^1(\Omega)}\right)\\
    &=\frac{1}{\lambda_{\rm{min}}}\LL(\widehat{\alpha}^{\LL}_{n})\leq\frac{1}{\lambda_{\rm{min}}}\inf_{\alpha\in\NN_n}\LL(\alpha)=\frac{1}{\lambda_{\rm{min}}}\inf_{\alpha\in\NN_n}\|A\alpha-F\|_{L^1(\Omega)}\\
    &\leq\frac{1}{\lambda_{\rm{min}}}\inf_{\alpha\in\NN_n}\left(\|A\alpha-A\alpha^*\|_{L^1(\Omega)}+\|A\alpha^*-F\|_{L^1(\Omega)}\right)\\
    &\leq\frac{\lambda_{\rm{max}}}{\lambda_{\rm{min}}}\inf_{\alpha\in\NN_n}\|\alpha-\alpha^*\|_{L^1(\Omega)}
    =\kappa(A)\inf_{\alpha\in\NN_n}\|\alpha-\alpha^*\|_{L^1(\Omega)}.
\end{align*}
\end{proof}

\begin{theorem}[Generalization error]\label{sym_gen_error}
    If Assumption \ref{f_ass} holds, we have
    \begin{equation}\label{symm_general_est}
        \EE\left[\|\widehat{\alpha}^{\LL}_n-\widehat{\alpha}^{\LL}_{n,M}\|_{L^1(\Omega)}\right]\lesssim
        \kappa(A)^{d/2} \mathcal{R}_M(\mathcal{F}^{\LL}_n)+\kappa(A)\inf_{\alpha\in\NN_n}\|\alpha-\alpha^*\|_{L^1(\Omega)}.
    \end{equation}
\end{theorem}
\begin{proof}
    By the use of Lemma \ref{matrix_thm} and Theorem \ref{Rade_thm}, we have
    \begin{align*}
        \EE\left[\|\widehat{\alpha}^{\LL}_n-\widehat{\alpha}^{\LL}_{n,M}\|_{L^1(\Omega)}\right]
        %&\leq\frac{1}{\lambda_{\min}}\EE\left[\|A\widehat{\alpha}^{\LL}_n-A\widehat{\alpha}^{\LL}_{n,M}\|_{L^1(\Omega)}\right]\\
        &\leq\frac{1}{\lambda_{\min}}\EE\left[\|A\widehat{\alpha}^{\LL}_n-F\|_{L^1(\Omega)}+\|A\widehat{\alpha}^{\LL}_{n,M}-F\|_{L^1(\Omega)}\right]\\
        &=\frac{1}{\lambda_{\min}}\EE\left[\LL(\widehat{\alpha}^{\LL}_n)+\LL(\widehat{\alpha}^{\LL}_{n,M})\right]\leq\frac{1}{\lambda_{\min}}\EE\left[\LL(\widehat{\alpha}^{\LL}_{n,M})\right]\\
        &\lesssim\frac{1}{\lambda_{\min}}\left(\EE\left[\LL(\widehat{\alpha}^{\LL}_{n,M})-\LL^M(\widehat{\alpha}^{\LL}_{n,M})\right]+\EE\left[\LL^M(\widehat{\alpha}^{\LL}_{n})\right]\right)\\
        &\lesssim\frac{1}{\lambda_{\min}}\RR_M(\mathcal{F}^{\LL}_n)+\frac{1}{\lambda_{\min}}\EE\left[\LL^M(\widehat{\alpha}^{\LL}_n)-\LL(\widehat{\alpha}^{\LL}_n)\right]+\frac{1}{\lambda_{\min}}\EE\left[\LL(\widehat{\alpha}^{\LL}_n)\right]\\
        &\lesssim\frac{1}{\lambda_{\min}}\RR_M(\mathcal{F}^{\LL}_n)+\frac{\lambda_{\max}}{\lambda_{\min}}\inf_{\alpha\in\NN_n}\|\alpha-\alpha^*\|_{L^1(\Omega)}.
    \end{align*}
Therefore, from \eqref{min_eigen} and \eqref{cond_lower_bound}, we obtain the desired inequality.
\end{proof}

Now, from Theorem \ref{symm_approx_error} and Theorem \ref{sym_gen_error}, we have the following estimates:
\begin{equation}\label{coeff_est_sym}
    \mathbb{E}\left[\|\alpha^*-\widehat{\alpha}^{\LL}_{n,M}\|_{L^1(\Omega)}\right]\lesssim\kappa(A)\inf_{\alpha\in\NN_n}\|\alpha-\alpha^*\|_{L^1(\Omega)}+\kappa(A)^{d/2}\RR_M(\mathcal{F}^{\LL}_n).
\end{equation}
The above estimate plays a pivotal role in the error analysis. For example, from the above inequality together with the classical finite element theory, we can estimate an error between the true solution $u$ and the solution prediction $u_{h,n,M}$ by the FEONet. More precisely, for the $(P \ell)$-finite element approximation (piecewise polynomial function degree less than or equal to $\ell\in\mathbb{N}$), by the optimal error estimate of finite element approximation and the elliptic regularity theory, it follows that
\[
    \|u-u_h\|_{L^1(\Omega;L^2(D))}\lesssim h^{\ell+1}\int_{\Omega}|u(\om)|_{H^{\ell+1}(D)}\dom\lesssim h^{\ell+1}\int_{\Omega}\|f(\om)\|_{H^{\ell-1}(D)}\dom= h^{\ell+1}\|f\|_{L^1(\Omega;H^{\ell-1}(D))}.
\]
Furthermore, for the finite element basis function $\phi_j^h$ on the shape-regular partition, it is straightforward to verify that $\|\phi^h_j\|_{L^2(D)}\approx h^{d/2}$. Therefore, as we did in \eqref{h_fix_ext}, we can obtain the following estimate:
\begin{equation}\label{total_est_1}
    \begin{aligned}
        \mathbb{E}\left[\|u-u_{h,n,M}\|_{L^1(\Omega;L^2(D))}
        \right]   
        &\lesssim h^{\ell+1}+\max_{1\leq i \leq N_h}\|\phi_j\|_{L^2(D)}N^{1/2}_h\mathbb{E}\left[\|\alpha^*-\widehat{\alpha}^{\LL}_{n,M}\|_{L^1(\Omega)}\right]\\
        %&\lesssim h^{\ell+1}+h^{-2}\inf_{\alpha\in\NN_n}\|\alpha-\alpha^*\|_{L^1(\Omega)}+h^{-d}\RR_M(\mathcal{F}^{\LL}_n)\\
        &\lesssim h^{\ell+1}+\kappa(A)\inf_{\alpha\in\NN_n}\|\alpha-\alpha^*\|_{L^1(\Omega)}+\kappa(A)^{d/2}\RR_M(\mathcal{F}^{\LL}_n)
    \end{aligned}
\end{equation}
Unlike the convergence obtained in Theorem \ref{main_thm_whole} where the error $u_h-u_{h,n,M}$ with fixed $h>0$ was addressed, here we do not fix $h>0$ and investigate the total error $u-u_{h,n,M}$ (instead of $u_h-u_{h,n,M}$), identifying the role of $h>0$ in the entire convergence. It is noteworthy from the estimate \eqref{total_est_1}, that the choice of small $h>0$ may not always guarantee a small error, and we need to choose a suitable $h>0$ to minimize the error for the predicted solution. In the next chapter, with some regularity assumption, we will further estimate the right-hand side of \eqref{total_est_1} in terms of $n$, $M\in\mathbb{N}$ to obtain the complete error estimate, which is the main goal of this paper.

\subsection{The Barron Space}
    In this section, we shall introduce a certain class of functions called the Barron space, to honor the seminal work of Barron concerning the mathematical analysis of a class of two-layer neural networks. It has been studied that this type of function can be well approximated by two-layer neural networks with dimension-independent convergence rates with respect to the width \cite{Barron_1, rade_upper_2}. We start with the so-called spectral Barron space which was Barron's original approach \cite{spectral_barron_2, spectral_barron_3, rade_upper_3}. For an integrable function $f$ defined on $\Omega\subset \R^m$, we define the quantity
\begin{equation}\label{barron_gamma}
    \gamma_s(f_e)=\int_{\R^m}(1+|\xi|)^s|\widehat{f_e}(\xi)|\,\mathrm{d}\xi,
\end{equation}
where $f_e\in L^1(\R^m)$ is an extension of $f$ to $\R^m$ and $\widehat{f_e}$ is the Fourier transform of $f_e$. We call the space of functions with bounded $\gamma_s(\cdot)$ as the Barron space. More precisely, we shall define
\[
    \mathcal{B}^s(\Omega):=\{f:\Omega\rightarrow\R:\gamma_s(f_e)<\infty\quad\text{for some extension of}\,\,f_e\,\,\text{of}\,\,f\},
\]
which is equipped with the norm
\[
    \|f\|_{\mathcal{B}^s(\Omega)}:=\inf_{f_e|_{\Omega}=f}\gamma(f_e),
\]    
where the infimum is taken over all extensions $f_e\in L^1(\R^m)$. Then one can show that $\mathcal{B}^s(\Omega)$ is continuously embedded in $H^s(\Omega)$ provided that $\Omega$ is bounded \cite{spectral_barron_2}. There are some results providing the characterization of the Barron Space. For example, a sufficiently smooth function with compact support is contained in the Barron space, which is presented in the following lemma (see, e.g., \cite{barron_ext}).
\begin{lemma}\label{diff_char}
    If $g\in C^{\beta}_0(\R^m)$ with $|\beta|>m/2+s$ for some $s\in\mathbb{N}$, then $g\in\mathcal{B}^s(\R^m)$ in the sense that
\begin{equation}\label{barron_char}
    \gamma_s(g)^2\lesssim\int_{\R^m}\left(|g|^2+|\partial^{\beta}g|^2\right)\dx<\infty.
\end{equation}
\end{lemma}
 More characterizations and properties can be found in various papers including \cite{Barron_1, Barron_2, spectral_barron_2}. Note that, if the function is defined in a bounded domain $\Omega$ rather than the whole space $\R^m$, we need to extend the functions from $\Omega$ to $\R^m$. This can be done by a continuous extension argument such as the Whitney extension theorem which is encapsulated in the following theorem (see, e.g., \cite{barron_ext}).
 \begin{lemma}\label{extension_thm}
    If $g\in C^{\beta}(\overline{\Omega})$, then for a closed set $\Gamma\subset\R^m$  satisfying $\Omega\subset\subset\Gamma$, there exists an extension $g_e\in C^{\beta}_0(\R^m)$ such that $g=g_e$ in $\overline{\Omega}$ and $g_e=0$ in $\R^m\setminus\Gamma$.         
 \end{lemma}
    
In modern machine learning theory, the Barron space is sometimes defined in a different way, using a probabilistic integral representation \cite{rade_upper_2, barron_def_1}. More precisely, let us consider a function $g:\Omega\rightarrow\R$ with the following integral representation
\begin{equation}\label{int_rep}
    g(\om)=\int_{\R\times\R^m\times\R}a\sigma(b\cdot\om+c)\rho({\rm{d}}a,{\rm{d}}b,{\rm{d}}c),
\end{equation}
where $\rho$ is a probability distribution on $\R\times\R^m\times\R$ and $\sigma$ is an activation function. Then we define the norm of such function as
\[
    \|g\|_{\mathcal{W}^s(\Omega)}=\inf_{\rho}\left(\mathbb{E}_{\rho}[|a|^s(\|b\|_1+|c|)^s]\right)^{1/s},
\]
where the infimum is taken over all $\rho$ with the representation \eqref{int_rep}. The probabilistic Barron space $\mathcal{W}^s(\Omega)$ is defined as the class of continuous functions which can be represented by \eqref{int_rep} with finite Barron norm $\|\cdot\|_{\mathcal{W}^s(\Omega)}$. This is less explicit than the Fourier-based characterization, but it is known to contain more functions that can be efficiently approximated by two-layer neural networks \cite{tight_1, tight_2}. 

The regularity we will mainly concern in this paper is $\mathcal{B}^2(\Omega)$, and hence, we shall denote $\mathcal{B}(\Omega)=\mathcal{B}^2(\Omega)$ henceforth.
In several papers, it was shown that $\mathcal{B}^{s+1}(\Omega)$ is continuously embedded into $\mathcal{W}^s(\Omega)$ for all $s\in\mathbb{N}$ \cite{Barron_1,FNM}, and $\mathcal{W}^1(\Omega)=\mathcal{W}^2(\Omega)=\mathcal{W}^3(\Omega)=\cdots=\mathcal{W}^{\infty}(\Omega)$ provided that $\sigma$ is the ReLU activation function \cite{rade_upper_2}. Therefore we will use the single notation $\mathcal{W}(\Omega)$ to denote the probabilistic Barron space. From the above properties, it follows that $\mathcal{B}(\Omega)\hookrightarrow\mathcal{W}(\Omega)$.

We begin with the following regularity result which identifies a sufficient condition for our target coefficient $\alpha^*$ to be contained in the Barron space. To do this, we need to assume the following.

\begin{assumption}\label{f_ass_2}
    For a compact set $\Omega\subset\R^m$, $f(x,\om)\in L^1(D;C^{\beta}(\overline{\Omega}))$ with $|\beta|>m/2+2$.    
\end{assumption}
\begin{proposition}\label{main_reg}
    Suppose that Assumption \ref{f_ass_2} holds. Then each target coefficient (finite element coefficient) $\alpha^*_j$ is contained in $\mathcal{B}(\Omega)$ for $j=1,\cdots, N_h$ with the estimate
    \[
    \|\alpha^*_j\|_{\mathcal{B}(\Omega)}\lesssim \kappa(A)^{d/2}.
    \]
\end{proposition}
\begin{proof}
By recalling the definition of the finite element coefficients $\alpha^*$, for each $j\in N_h$ 
\[
    \alpha^*_j(\om)=\tilde{a}^j_1F_1(\om)+\cdots+\tilde{a}^j_{N_h}F_{N_h}(\om)=\int_Df(x,\om)[\tilde{a}^j_1\phi_1(x)+\cdots+\tilde{a}^j_{N_h}\phi_{N_h}(x)]\dx=:\int_Df(x,\om)B_j(x)\dx,
\]
where $(\tilde{a}^j_1,\cdots,\tilde{a}^j_N)$ denotes the $j$-th row  of $A^{-1}$. Since the largest component of the inverse of $A$ is bounded above by $\lambda_{\min}^{-1}$ (see, e.g., \cite{mat_anal}), from \eqref{inv_min_eigen}, we see that $|B_j|\lesssim \kappa(A)^{d/2}$. 

Next, by Assumption \ref{f_ass_2} and Lemma \ref{extension_thm}, there exists an extension $f_{e_x}\in C^{\beta}_0(\R^m)$ for each $x\in D$. We then define an extension of $\alpha^*_j$ by
    \[
        \alpha^*_{j,e}(\om)=\int_D f_{e_x}(x,\om)B_j(x)\dx.
    \]
By Fubini's theorem, we have
\begin{align*}
    \widehat{\alpha^*_{j,e}}(\xi) &=\int_{\R^m}\bigg[\int_Df_{e_x}(x,\om)B_j(x)\dx\bigg]e^{-i\xi\cdot\om}\dom\\
    &\lesssim \kappa(A)^{d/2}\int_D\int_{\R^m}f_{e_x}(x,\om)e^{-i\xi\cdot\om}\dom\dx=\kappa(A)^{d/2}\int_D{\widehat{f_{e_x}}}(x,\xi)\dx.
\end{align*}
    
Again, by Fubini's theorem together with \eqref{barron_char} and the continuity of the extension $e_x$, it follows that
\begin{align*}
        \gamma_s(\alpha^*_{j,e})
        &=\int_{\R^m}(1+|\xi|)^s|\widehat{\alpha^*_{j,\tilde{e}}}(\xi)|\,{\rm{d}}\xi\lesssim \kappa(A)^{d/2}\int_{\R^m}(1+|\xi|)^s\int_D\bigg|\widehat{f_{e_x}}(x,\xi)\bigg|\dx\,{\rm{d}}\xi\\
        &\lesssim \kappa(A)^{d/2}\int_D\int_{\R^m}(1+|\xi|)^s|\widehat{f_{e_x}}(x,\xi)|\,{\rm{d}}\xi\dx\lesssim \kappa(A)^{d/2}\int_D\gamma_s(f_{e_x})\dx\lesssim \kappa(A)^{d/2}\int_D\|f\|_{C^{\beta}(\overline{\Omega})}\dx.
    \end{align*}
From Assumption \ref{f_ass_2}, we can conclude that $\alpha^*_{j}\in B(\Omega)$ for all $j=1,\cdots,N_h$.
\end{proof}

Based on Proposition \ref{main_reg}, henceforth, we will assume that our target coefficient satisfies $\alpha^*_j\in B(\Omega)$ for all $1\leq j\leq N_h$. Next, let us define an ansatz space for approximation. In the previous sections, we used the general class of feed-forward neural networks as approximators. On the other hand, here we use a set of two-layer ReLU neural networks, which is known to effectively approximate the function of Barron type. Specifically, the class of feed-forward neural networks $\NN_n$ defined in Section \ref{subsec:nn} now denotes the family of two-layer ReLU networks from $\Omega$ into $\R^{N_h}$, which contains $n$ neurons in the hidden layer. Also, we shall use the following class of scalar-valued neural networks:
\[
    \NN_{n,2}:=\bigg\{\alpha(\om;\theta)=\frac{1}{n}\sum^n_{j=1}a_j\sigma(b_j\cdot\om+c_j):\theta=(a_j,b_j,c_j)^n_{j=1},\,\,a_j,\,c_j\in\R,\,\,\,b_j\in\R^m\,\,\,\text{for all}\,\,\,1\leq j\leq n\bigg\}.  
\]
Furthermore, for a two-layer neural network with the parameter $\theta=(a_j,b_j,c_j)^n_{j=1}$, we define the so-called path norm by
\[
    \|\theta\|_{\mathcal{P}}:=\frac{1}{n}\sum^n_{j=1}|a_j|(\|b_j\|_1+|c_j|).    
\]
Then we define the main ansatz class in this section which consists of the two-layer ReLU neural networks with the parameters whose path norms are bounded by the Barron norm of the target coefficients:
\[
    \NN^*_{n,2}:=\bigg\{\alpha(\om;\theta)\in\NN_n:\|\theta\|_{\mathcal{P}}\leq 2\max_{1\leq j\leq N_h}\|\alpha_j^*\|_{\mathcal{W}(\Omega)}\bigg\}.  
\]
It is known that if the target function is contained in the Barron space, it can be well approximated by two-layer neural networks without the curse of dimensionality which is presented in the following theorem (see, for example, \cite{Barron_1, rade_upper_2}).
\begin{theorem}\label{barron_approx}
    Let $g\in\mathcal{W}(\Omega)$ and $n\in\mathbb{N}$. Then there exists a two-layer ReLU neural network $g_n(\cdot;\theta)\in\NN_{n,2}$ with $\|\theta\|_{\mathcal{P}}\leq2\|g\|_{\mathcal{W}(\Omega)}$ satisfying
    \[
        \|g(\cdot)-g_n(\cdot;\theta)\|^2_{L^2(\Omega)}\leq\frac{3\|g\|^2_{\mathcal{W}(\Omega)}}{n}.
    \]
\end{theorem}

Note that the above theorem is for the scalar-valued neural networks. Let us describe the way to apply Theorem \ref{barron_approx} to our case, where the vector-valued neural networks are considered. For the target coefficient  $\alpha^*(\om)=(\alpha^*_i(\om))^{N_h}_{i=1}$ with $\alpha^*_i\in\mathcal{W}(\Omega)$, by Theorem \ref{barron_approx}, there exists two-layer neural networks 
\[
    \alpha^*_{i,n}(\om)=\frac{1}{n}\sum^n_{j=1}a^i_j\sigma(b^i_j\cdot\om+c^i_j)\in\NN^*_{n,2}
\]
for each $i=1,\cdots,\N_h$, satisfying
\[
    \|\alpha^*_i-\alpha^*_{i,n}\|^2_{L^2(\Omega)}\leq\frac{3\|\alpha^*_i\|^2_{\mathcal{W}(\Omega)}}{n}.
\]
Now we define a (vector-valued) two-layer ReLU neural network $\widehat{g}:\Omega\rightarrow\R^{N_h}$ with $nN_h$ nodes such that the weight matrix and the bias vector of the hidden layer is defined by
\[
W^1=(b^1_1,\cdots, b^1_n,\cdots,b^{N_h}_1,\cdots b^{N_h}_n)^T\in\R^{nN_h\times n}
\quad{\rm{and}}\quad
b^1=(c^1_1,\cdots,c^1_n\cdots,c^{N_h}_1,\cdots,c^{N_h}_n)^T\in\R^{nN_h}.
\]
The weight matrix of the second layer is defined by
\[
W^2 = 
\begin{pmatrix}
a^1_1\cdots a^1_n & 0\cdots0 & 0\cdots0 & \cdots &0\cdots0 \\
0\cdots0 &a^2_1\cdots a^2_n & 0\cdots0 & \cdots & 0\cdots0 \\
\vdots & \vdots & \vdots & \ddots & \vdots \\
0\cdots0  & 0\cdots0  & 0\cdots0 &\cdots& a^{N_h}_1\cdots a^{N_h}_n 
\end{pmatrix}.
\]
Then we have
\begin{align*}
    \inf_{g\in\NN_{nN_h}}\|\alpha^*-g\|_{L^1(\Omega)}
    &\lesssim\|\alpha^*-\widehat{g}\|_{L^2(\Omega)}=\left(\|\alpha^*_1-\alpha^*_{1,n}\|^2_{L^2(\Omega)} + \cdots + \|\alpha^*_{N_h}-\alpha^*_{N_h,n}\|^2_{L^2(\Omega)}\right)^{\frac{1}{2}}\\
    %&\lesssim\frac{1}{\sqrt{n}}\left(\|\alpha^*_1\|^2_{\mathcal{W}(\Omega)}+\cdots + \|\alpha^*_{N_h}\|^2_{\mathcal{W}(\Omega)}\right)^{\frac{1}{2}}
    &\lesssim \frac{1}{\sqrt{n}}\left(\|\alpha^*_1\|^2_{\mathcal{B}(\Omega)}+\cdots + \|\alpha^*_{N_h}\|^2_{\mathcal{B}(\Omega)}\right)^{\frac{1}{2}}\lesssim \frac{\kappa(A)^{3d/4}}{\sqrt{n}}.
\end{align*}
Therefore, we obtain the following result.
\begin{proposition}
        Suppose that Assumption \ref{f_ass_2} holds. Then we have
        \begin{equation}\label{barron_approx_1}
        \inf_{\alpha\in\NN_n}\|\alpha^*-\alpha\|_{L^1(\Omega)}\lesssim\frac{\kappa(A)^d}{\sqrt{n}}.
        \end{equation}
\end{proposition}

Next, we shall investigate the generalization error. To do this, we will use the following result \cite{rade_upper_2, rade_gen}.

\begin{theorem}\label{barron_general}
    For the function class $\mathcal{F}_{Q,1}=\{g(\cdot;\theta)\in\NN_{n,2}:\|\theta\|_{\mathcal{P}}\leq Q\}$ and $\mathcal{F}_{Q,2}=\{g\in\mathcal{W}:\|g\|_{\mathcal{W}}\leq Q\}$, we have
    \begin{equation}\label{gen_rade_thm}
        \RR_M(\mathcal{F}_{Q,j})\leq 2Q\sqrt{\frac{2\log(2m+2)}{M}}\quad{\rm{for}}\,\,\,j=1,2.
    \end{equation}
\end{theorem}
Let us write $A=(a_{ij})$ and $a_{\max}=\max_{i,j}|a_{ij}|$. Since the finite element matrix $A$ is symmetric and positive-definite, there exists some $k\in\mathbb{N}$ such that $|a_{ij}|\leq a_{kk}$ for any $i$, $j=1,2,\cdots, N_h$ (see, e.g., \cite{mat_anal}). Therefore, from the estimate \eqref{inv_min_eigen}, it follows that $a_{\max}\leq |a_{kk}|= e_k^TAe_k\leq\lambda_{\max}\lesssim\kappa(A)^{1-d/2}$, where $e_k$ denotes the $k$-th standard basis. In order to apply Theorem \ref{barron_general} to our case, let us first define the following function classes:
\begin{align*}
    &\mathcal{F}_n:=\{|A\alpha-F|:\alpha=(\alpha_j)^{N_h}_{j=1},\,\,\,\alpha_j\in\NN^*_{n,2}\,\,{\rm{for}}\,\,j=1,\cdots,N_h\},\\
    &\mathcal{F}_{n,1}:=\{|A\alpha|:\alpha=(\alpha_j)^{N_h}_{j=1},\,\,\,\alpha_j\in\NN^*_{n,2}\,\,{\rm{for}}\,\,j=1,\cdots,N_h\},\\
    &\mathcal{F}^{i,j}_{n,1}:=\{|a_{ij}g|:g\in \NN^*_{n,2}\},\quad\mathcal{F}^{\max}_{n,1}:=\{a_{\max}g:g\in\NN^*_{n,2}\}.
\end{align*}
We first note that $\mathcal{F}^{\max}_{n,1}\subset\{g\in\NN_{n,2}:\|g\|_{\mathcal{P}}\leq 2a_{\max}\|\alpha^*\|_{\mathcal{W}(\Omega)}\}$. Since the Rademacher complexity of a set of a single function is zero, by Theorem \ref{barron_general}, Proposition \ref{main_reg} and Talagrand’s contraction principle (see, e.g., \cite{Rade_2}), we have that
\begin{equation}\label{barron_gen_est}
    \RR_M(\mathcal{F}_n)\leq \RR_M(\mathcal{F}_{n,1})\leq \sum^{N_h}_{i,j=1}\RR_M(\mathcal{F}^{i,j}_{n,1})\leq(N_h)^2\mathcal{R}_M(\mathcal{F}^*_{\max})\lesssim\frac{\kappa(A)^{1+d}}{\sqrt{M}}.
\end{equation}
\begin{comment}
For the second term, under Assumption \ref{f_ass_2}, we note that for any element $g\in\mathcal{F}^j_{n,2}$, $\|g\|_{\mathcal{B}(\Omega)}\leq\|f\|_{L^1(D;C^{\beta}(\overline{\Omega}))}$ for all $1\leq j\leq N_h.$ By Theorem \ref{barron_general}, this leads us to the estimate
\[
     ({\rm{II}})\leq \frac{h^{-d}}{\sqrt{M}}.
\]

Therefore, we finally have that
\begin{equation}\label{barron_gen_est}
    \mathcal{R}(\mathcal{F}^*_n)\lesssim\frac{h^{-d}+h^{-2-3d}}{\sqrt{M}}  
\end{equation}
\end{comment}
Therefore, from \eqref{total_est_1}, \eqref{barron_approx_1} and \eqref{barron_gen_est}, we finally obtain the following theorem, which is the main result of this paper.
\begin{theorem}\label{final_thm}
    Let Assumption \ref{f_ass_2} hold. If we use the $(P \ell)$-finite element approximation, the predicted solution $u_{h,n,M}$ by the FEONet satisfies the following error estimate:
    \begin{equation}\label{final_err_est}
        \mathbb{E}\left[\|u-u_{h,n,M}\|_{L^1(\Omega;L^2(D))}
        \right]  \lesssim h^{\ell+1}+\frac{\kappa(A)^{1+d}}{\sqrt{n}}+\frac{\kappa(A)^{1+3d/2}}{\sqrt{M}}.
    \end{equation}
%{\color{red}
%    \[
%        {\rm{(Original}}\,\,{\rm{version)}}\,\,\,
%        \mathbb{E}\left[\|u-u_{h,n,M}\|_{L^1(\Omega;H^1(D))}
%        \right]  \lesssim h+h^{-2-4d}\sqrt{1+h^{-2}}\left(\frac{1}{\sqrt{n}}+\frac{1+h^{2+2d}}{\sqrt{M}}\right).
%    \]
%}
\end{theorem}
\begin{remark}
    As described earlier, in the method we propose, we first choose the finite element parameter $h>0$ and make a triangulation. During this process, the first term in the error \eqref{final_err_est} can be reduced as desired, and better convergence can be achieved when using higher-order methods. After $h>0$ is fixed and basis functions are set, we proceed with the neural network approximation for the target coefficients. Since $\kappa(A)\approx h^{-2}$, if $h>0$ is small, the second and third terms of the error in \eqref{final_err_est} may increase. However, by increasing the number of neurons in the hidden layer ($n\rightarrow\infty$) and using more training samples for training ($M\rightarrow\infty$), we can reduce these error terms as much as we want. Additionally, by using the preconditioning techniques, we can significantly reduce the condition number. This provides a theory-guided strategy to reduce the second and third error terms in \eqref{final_err_est}, ensuring faster convergence. This will be explicitly confirmed through numerical experiments in the next section.
\end{remark}

\section{Numerical experiments}\label{sec:num_exp}
In this section, we shall perform some numerical experiments to confirm the theoretical findings in the previous sections. Motivated by \cite{bar2019learning}, we randomly generate input samples (in this paper, external forces) of the form
\begin{equation}\label{input_type}
f(x)=n_1\sin(m_1\cdot x)+n_2\cos(m_2\cdot x)\quad x\in D,
\end{equation}
where the random parameters $n_1$, $n_2$, $m_1$ and $m_2$ are sampled from uniform distributions. As described in Section \ref{subsec:desc}, we train the FEONet using these random samples. Note that we don't need any precomputed input-output $(f, u)$ pairs for training. We then evaluate the performance of the trained model using different random samples that were not used in the training. For the corresponding true solutions for the test data, we computed the finite element solutions on a sufficiently fine mesh, using the finite element software package FEniCS \cite{fenics}. All the computations were conducted using the Intel Xeon Cascade Lake (Gold 6226R) processor and TESLA V100 GPU.

\subsection{Convergence against the number of training samples and model size}\label{sec_num_sample}

\begin{figure}
\begin{center}
\includegraphics[width=\textwidth]{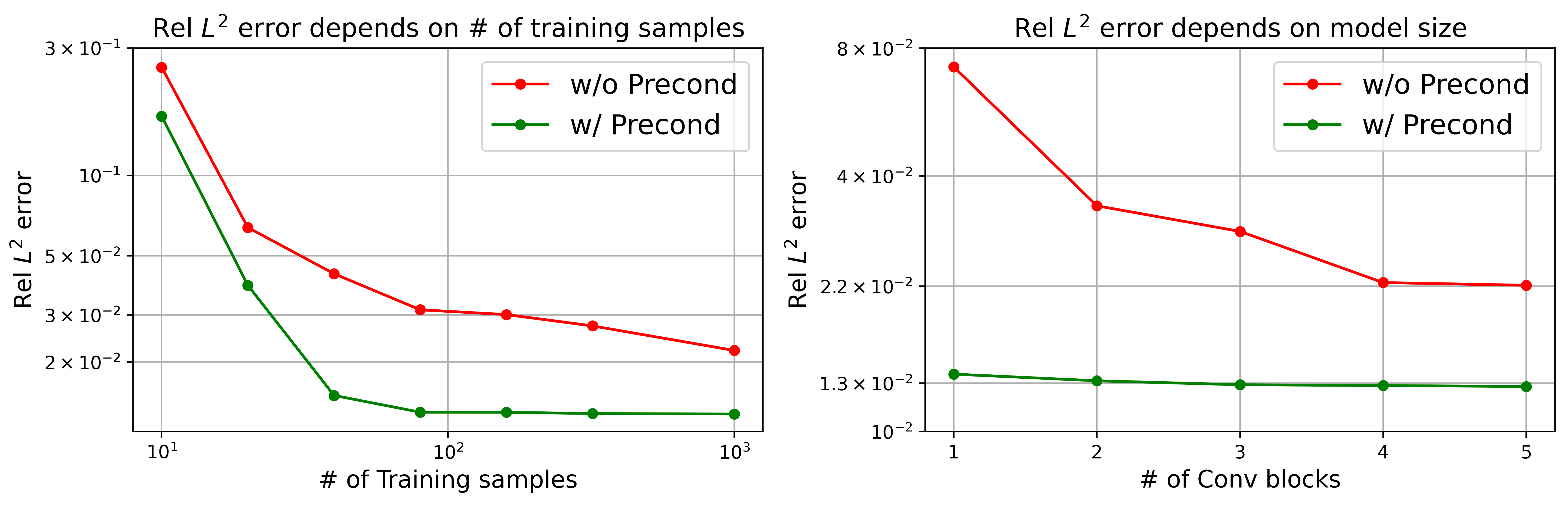}
\end{center}
\caption{The relative $L^2$ errors resulting from varying the number of training samples and the model size.}
\label{fig:sample_modelsize}
\end{figure}

We first demonstrate that for fixed $h>0$, the error $\|u_h-u_{h,n,M}\|$ decreases as the model size $n$ and the number of input samples $M$ increase, which was theoretically proved in Theorem \ref{main_thm_whole}. To do this,
we consider the 2D Poisson equation within the domain $D$, which is a square with a hole (see the second domain in Figure \ref{fig:meshes}), i.e.,
\begin{equation}\label{eq:poisson}
    \begin{aligned}
        - \Delta u(x,y)&=f(x,y), && (x,y) \in D,\\
        u(x,y)&=0, && (x,y) \in \partial D.
    \end{aligned}
\end{equation}
Figure \ref{fig:sample_modelsize} illustrates the experimental results, where we can see a relationship between relative errors and the approximation parameters; the model size $n\in\mathbb{N}$ and the number of training samples $M\in\mathbb{N}$. We use the convolutional neural networks (CNN) as a baseline model and train the model with the gradually increasing number of convolutional blocks. We also conduct the training with a varying number of training samples from 10 to 20, 40, 80, 160, 320, and up to 1000. As we expected, we can confirm the decreasing tendency of errors as depicted with red lines in Figure \ref{fig:sample_modelsize}.

Secondly, what we also proved in the theory was that the convergence depends on the condition number $\kappa(A)$ (e.g. Theorem \ref{approx_error}, Theorem \ref{coef_gen_err}, \eqref{coeff_est_sym} and \eqref{total_est_1}). This motivates us to train the FEONet with the preconditioned loss
\begin{equation}\label{precond_loss}
    \mathcal{K}(\alpha)=\|P^{-1}A\alpha(\om)-P^{-1}F(\om)\|_{L^1(\Omega)}\quad{\rm{and}}\quad \mathcal{K}^M(\alpha)=\frac{|\Omega|}{M}\sum_{i=1}^M|P^{-1}A\alpha(\om_i)-P^{-1}F(\om_i
    )|,
\end{equation}
so that all of $\kappa(A)$ in the theoretical results can be replaced by the term $\kappa(P^{-1}A)$ which is known to be smaller. In our experiment, for the preconditioning, we used the Sparse Approximate Inverse (SPAI) preconditioner \cite{chow1998approximate}. The experimental results for this are demonstrated in Figure \ref{fig:sample_modelsize} with green lines. As we expected from theory, the errors are reduced and it has been shown that we can train the model with smaller numbers of training samples and convolutional blocks, and hence we can improve the training efficiency if we use the preconditioning. This provides us with a theory-guided classical numerical analysis strategy to improve the performance of the proposed machine-learning model.

\subsection{Convergence against the number of element}

\begin{figure}[t]
\begin{center}
\includegraphics[width=\textwidth]{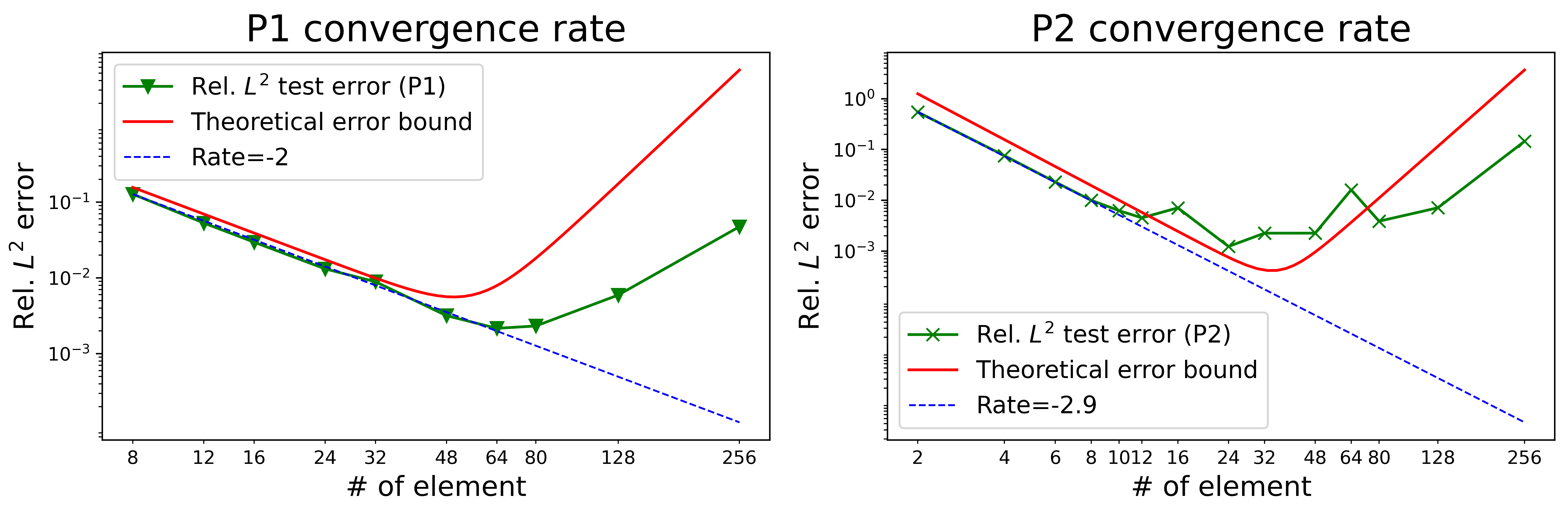}
\end{center}
\caption{The relative $L^2$ errors against the number of elements.}
\label{fig:rate}
\end{figure}
% \begin{figure}[H]
% \begin{center}
% \includegraphics[width=0.8\textwidth]{images/3_1d_convergence_large.png}
% \end{center}
% \caption{The relative $L^2$ errors by varying the number of elements.}
% \label{fig:rate}
% \end{figure}

Next, we shall confirm the theoretical result presented in Theorem \ref{final_thm}. Note that the input samples of type \eqref{input_type} satisfy Assumption \ref{f_ass_2}. For this experiment, we consider the convection-diffusion equation
\begin{equation}\label{eq:conv_diff}
  \begin{aligned}
    -0.1 u_{xx}-u_x&=f(x), && x \in [-1,1],\\
    u(-1)&=u(1)=0.
  \end{aligned} 
\end{equation} 
The experimental results concerning the quantitative relationship between test errors and the number of elements are depicted in green in Figure \ref{fig:rate}; one for piecewise linear approximation and the other for piecewise quadratic approximation. As we can see in Figure \ref{fig:rate}, up to a certain point, as the degree of freedom increases, the relative $L^2$ error decreases following the convergence rate which is known in the classical theory of FEM ($-2$ for P$1$ approximation and $-3$ with P$2$ approximation), as indicated by dotted lines. But after that point, we can see that the error increases again. This phenomenon was predicted in the theoretical result presented in Theorem \ref{main_thm_whole}. Thanks to the first term in \eqref{final_err_est}, the error goes down when it reaches a local minimum. However, as the number of elements further increases, the condition number $\kappa(A)$ goes up, and hence, so does the relative $L^2$ error. The theoretical error bounded for both P$1$ and P$2$ obtained in Theorem \ref{main_thm_whole} are drawn in red in Figure \ref{fig:rate}, and we can confirm that the experimental results have a similar tendency as the theoretically predicted results.

Secondly, according to Theorem \ref{main_thm_whole}, when the number of elements is large, we can reduce the total error by using the preconditioning technique. As discussed in Section \ref{sec_num_sample}, we compare the results for the original FEONet, and the one using the preconditioned loss \eqref{precond_loss}, and the experimental result is presented in Figure \ref{fig:solution_profile}. When the number of elements is relatively small, the preconditioning slightly improves the performance. However, for the case when the number of elements is large, using the preconditioner for the training significantly improves the overall performance of the model. This also matches with the theoretical findings in the previous sections.

\section{Conclusion}
In this paper, we have investigated the convergence of approximate solutions predicted by the FEONet proposed in \cite{FEONet}, and derived an error estimate, simultaneously examining the whole approximation parameters $h>0$ and $n$, $M\in\mathbb{N}$. Based on the eigenvalue estimates from the classical FEM theory, we proved the convergence of the numerical solution for general linear second-order PDEs, and the convergence depends on the condition number of the finite element matrix. Furthermore, for a self-adjoint case, a novel regularity theory for the neural network approximation was proposed, from which we obtained the complete error estimate. Finally, we conducted some numerical experiments which support the theoretically shown results in the paper.

\begin{figure}
\begin{center}
\includegraphics[width=0.9\textwidth]{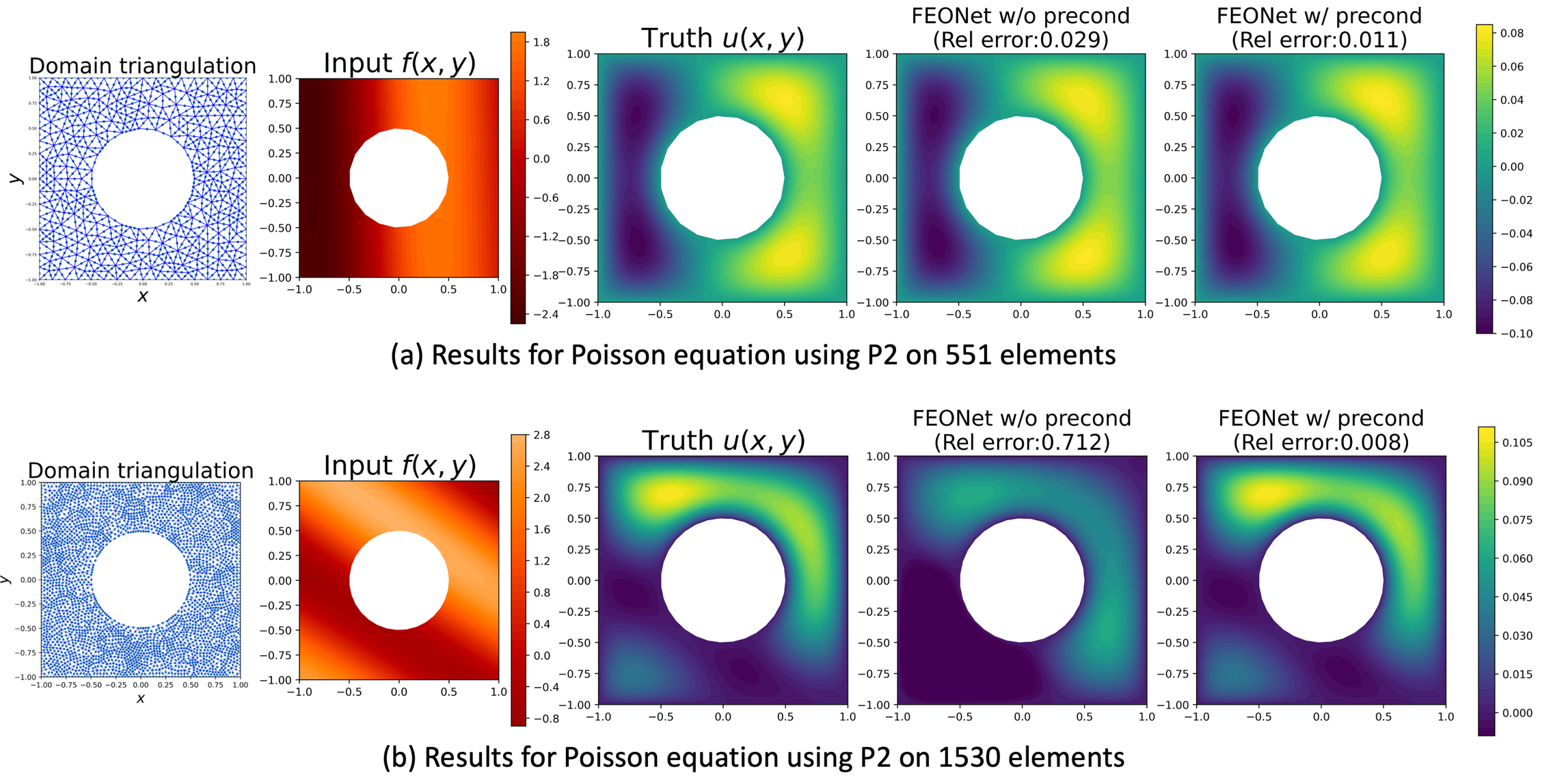}
\end{center}
\caption{The FEONet prediction for 2D Poisson equation without or with preconditioning.}
\label{fig:solution_profile}
\end{figure}

An interesting future research direction is to study the optimization error. Since the neural tangent kernel (see e.g., \cite{NTK}) of the FEONet also depends on the finite element matrices, again from the eigenvalue estimates for FEM, we may theoretically analyze the corresponding optimization error. The analysis of the FEONet for nonlinear equations is also intriguing. In \cite{FEONet}, we showed that the FEONet can also learn a solution operator for nonlinear equations. This ability becomes particularly highlighted if we compare the FEONet with the classical FEM, since we don't need any iterative schemes to predict solutions, allowing us to make a real-time solution prediction for nonlinear equations. These topics are of independent interest, and will be addressed in the forthcoming papers.

\bibliography{references}
\bibliographystyle{abbrv}

%\section*{Appendix. Proof of Theorem \ref{thm:simple_system} and \ref{thm:geometric_length}}
\begin{comment}
\section*{Appendix: Additional Figure}
\begin{figure}[H]
\begin{center}
\includegraphics[width=\textwidth]{images/precond.png}
\end{center}
\caption{Results for 2D Poisson equation using the FEONet without or with preconditioning.}
\label{fig:solution_profile}
\end{figure}

Figure \ref{fig:solution_profile} shows the result for the operator learning and Figure \ref{fig:preconditioning} shows the rlearning procedure when the precondition is used. 
\begin{figure}[H]
\begin{center}
\includegraphics[width=\textwidth]{images/2_preconditioning.png}
\end{center}
\caption{The relative $L^2$ errors by varying the number of training samples and the model size.}
\label{fig:preconditioning}
\end{figure}
\end{comment}

\end{document}